\PassOptionsToPackage{quiet}{xeCJK} 

\documentclass[12pt,reqno]{amsart} 

\usepackage{mathrsfs,amsfonts,amsthm,amsmath,amssymb,mathtools} 
\usepackage[dvipsnames]{xcolor} 

\usepackage{enumitem}                 
\usepackage[left]{lineno}             

\usepackage[numbers,sort&compress]{natbib}  

\usepackage{graphics}
\usepackage{graphicx}   
\graphicspath{{images/}} 
\usepackage{subfig}     
\usepackage{rotating}

\usepackage{hyperref} 
\usepackage{enumitem} 
\usepackage{cleveref}

\hypersetup{
  colorlinks   =  true, 
  urlcolor     = black, 
  linkcolor    = black, 
  citecolor   = black 
}

\theoremstyle{plain}
\newtheorem{theorem}{Theorem}[section] 
\newtheorem{conjecture}[theorem]{Conjecture}
\newtheorem{corollary}[theorem]{Corollary}

\newtheorem{proposition}[theorem]{Proposition}
\newtheorem{lemma}[theorem]{Lemma}

\newtheorem*{problem2}{\bf Problem A}

\theoremstyle{remark} 
\newtheorem{remark}[theorem]{Remark}
\newtheorem{example}[theorem]{Example}

\makeatletter
    \def\subsection{\@startsection{subsection}{2}%
    \z@{.5\linespacing\@plus.7\linespacing}{.3\linespacing}%
    {\normalfont\bfseries}}
\makeatother 

\makeatletter
    \newcommand{\LeftEqNo}{\let\veqno\@@leqno}
    
\makeatother

\numberwithin{equation}{section}

\def\B{\mathcal B}

\def\G{\mathcal G}
\def\H{\mathcal H}

\def\M{\mathcal M}
\def\N{\mathcal N}

\begin{document}

\




\title{Density of irreducible operators in the trace-class norm}

\author[J. Fang]{Junsheng Fang}
\address{Junsheng Fang, School of Mathematical Sciences, Hebei Normal University, Shijiazhuang, 050024, China}
\email{jfang@hebtu.edu.cn}

\author[C. Jiang]{Chunlan Jiang}
\address{Chunlan Jiang, School of Mathematical Sciences, Hebei Normal University, Shijiazhuang, 050024, China}
\email{cljiang@hebtu.edu.cn}

\author[M. Ma]{Minghui Ma}
\address{Minghui Ma, School of Mathematical Sciences, Dalian University of Technology, Dalian, 116024, China}
\email{minghuima@dlut.edu.cn}

\author[J. Shen]{Junhao Shen}
\address{Junhao Shen, Department of Mathematics \& Statistics, University of New Hampshire, Durham, 03824, US}
\email{Junhao.Shen@unh.edu}

\author[R. Shi]{Rui Shi}
\address{Rui Shi, School of Mathematical Sciences, Dalian University of Technology, Dalian, 116024, China}
\email{ruishi@dlut.edu.cn, ruishi.math@gmail.com}
\thanks{This article was partly supported by Tianyuan Mathematics Research Center. Chunlan Jiang and Junsheng Fang were partly supported by the Hebei Natural Science Foundation (No.A2023205045). Minghui Ma was partly supported by the China Postdoctoral Science Foundation (No.2025M783077) and the Postdoctoral Fellowship Program of CPSF (No.GZC20252022). Rui Shi 
was partly supported by NSFC (No.12271074) and a research fund from the Math Department of HEBNU for visiting scholars. }

\author[T. Wang]{Tianze Wang}
\address{Tianze Wang, School of Mathematical Sciences, Dalian University of Technology, Dalian, 116024, China}
\email{swan0108@mail.dlut.edu.cn}



\begin{abstract}
	In operator theory, a long-standing open problem concerns the density of irreducible operators on a separable complex Hilbert space $\mathcal{H}$ with respect to the trace-class norm. This line of research can be traced back to Halmos’ work on the density of irreducible operators in the operator norm topology.

	In this paper, for a large family of operators in $\mathcal{B}(\mathcal{H})$, we give this problem an affirmative answer.
	The result is derived from a combination of techniques in both operator theory and operator algebras. We also discover that there is a strong connection between this problem and an operator-theoretical problem related to type $\mathrm{II}_1$ factors. 
    
    Moreover, we reduce the above problem to the following form. For each operator $T$ in $\mathcal{B}(\mathcal{H})$ and every $\varepsilon>0$, is there a trace-class operator $K$ with $\|K\|_1<\varepsilon$ such that $T+K$ is a direct sum of at most countably many irreducible operators?

\end{abstract}

\maketitle

\section{Introduction}\label{section introduction}

Throughout this paper, let $\mathcal{H}$ be a separable infinite-dimensional complex Hilbert space, and let $\mathcal{B}(\mathcal{H})$ denote the algebra of all bounded linear operators on $\mathcal{H}$.
Recall that an operator $T$ in $\mathcal{B}(\mathcal{H})$ is {\em irreducible} if it has no nontrivial reducing subspaces.
That is to say, if $P$ is a projection (i.e., $P^2=P=P^*$) in $\mathcal{B}(\mathcal{H})$ such that $PT=TP$, then either $P=0$ or $P=I$.
By definition, the irreducibility of operators is invariant up to unitary equivalence.
We present the long-standing problem as follows.

\begin{problem2}\label{prob A}
For each operator $T$ in $\mathcal{B}(\mathcal{H})$ and $\varepsilon>0$, is there a trace-class operator $K$  in $\mathcal{B}(\mathcal{H})$  with $\|K\|_1<\varepsilon$ such that $T+K$ is irreducible?
\end{problem2}

Problem A can be traced back to a result of Paul Halmos in \cite{Hal}.
Next, we briefly recall Halmos' result about irreducible operators, the contributions of various authors related to Problem A, the main techniques applied previously, and the reason why the Weyl-von Neumann theorem fails to contribute to the solution of Problem A.

\subsection{Density problem of irreducible operators} 

By definition, irreducible operators can be viewed as atoms to construct operators in $\mathcal{B}(\mathcal{H})$.
Thus, it is natural to explore how large the set of irreducible operators is in the topological sense.
In the operator norm topology, Paul Halmos proved that irreducible operators form a dense $G_\delta$ subset of $\mathcal{B}(\mathcal{H})$ in \cite{Hal}.
Later, Heydar Radjavi and Peter Rosenthal gave a short proof in \cite{Rad}.
It turns out that, on considering the operator norm density of the set of irreducible operators, one needs the classical form of the spectral theorem for self-adjoint operators and matrix-construction techniques.

In the last paragraph of \cite[Section 1]{Hal}, Ronald Douglas observed that by virtue of the Weyl-von Neumann theorem, {\em Halmos' density theorem is also true in the sense of Hilbert-Schmidt approximation}.
To improve the result with the Schatten $p$-norm \cite{Schatten1960}, one needs a type of the Weyl-von Neumann theorem for self-adjoint operators as a key technique.
In the following part, we denote by $\Vert\cdot\Vert_p$-norm the Schatten $p$-norm for $p\geqslant 1$.
Note that the Schatten $2$-norm is the Hilbert-Schmidt norm, while the Schatten $1$-norm is the trace-class norm.

The classical {\em Weyl-von Neumann theorem} for self-adjoint operators in $\mathcal{B}(\mathcal{H})$ due to Hermann Weyl \cite{Weyl} and John von Neumann \cite{Von2} states that every self-adjoint operator is diagonalizable up to an arbitrarily small Hilbert-Schmidt perturbation.

In \cite{Kuroda}, Shige Toshi Kuroda improved the Weyl-von Neumann theorem by proving that every self-adjoint operator in $\mathcal{B}(\mathcal{H})$ is diagonalizable up to an arbitrarily small $\Phi$-norm perturbation, where by $\Phi$-norm we denote a unitarily invariant norm not equivalent to the trace-class norm.
Note that the $\Vert\cdot\Vert_p$-norm serves as a candidate for such a unitarily invariant norm for every $p> 1$. On the other hand, Dan Voiculescu \cite{Voi79} established a type of Weyl-von Neumann theorem for an $n$-tuple of commuting self-adjoint operators with his $\mathcal{C}^{-}_{n}$-perturbation ($n\geqslant 2$), by using his remarkable noncommutative Weyl-von Neumann theorem \cite{Voi76}. This line of research later evolved into a profound theory of normed ideal perturbations \cite{Voi81,DA2014,Voi2016,Voi2018,Voi2019,Voi2024,Voi2025},  which is also closely connected to the Kato-Rosenblum theorem \cite{Kato, Rosenblum}.

According to the Weyl-von Neumann-Kuroda theorem in \cite{Kuroda} and techniques of H.\,Radjavi and P.\,Rosenthal in \cite{Rad}, Domingo Herrero proved in \cite[Lemma 4.33]{Herrero1} that {\em the set of irreducible operators is $\Vert\cdot\Vert_p$-norm dense in $\mathcal{B}(\mathcal{H})$ for every $p>1$}.

\subsection{Schatten 1-norm perturbations of self-adjoint operators}

The line of research on the density of the set of irreducible operators with respect to the $\Vert \cdot \Vert_1$-norm would be intact if the Weyl-von Neumann theorem held for the $\Vert \cdot \Vert_1$-norm.
But, with respect to the $\Vert\cdot\Vert_1$-norm, a large family of self-adjoint operators fails to be diagonalizable up to trace-class perturbation. According to \cite{Kato, Rosenblum}, Tosio Kato and Marvin Rosenblum (independently) showed that, up to unitary equivalence, the {\em $(spectrally)$ absolutely continuous part} of a self-adjoint operator in $\mathcal{B}(\mathcal{H})$ is stable under self-adjoint trace-class perturbations.
Additionally, in \cite{CP}, Richard Carey and Joel Pincus showed that each purely singular self-adjoint operator in $\mathcal{B}(\mathcal{H})$  is a small trace-class perturbation of a diagonal operator.

By the Kato-Rosenblum theorem, the method in the proof of \cite[Lemma 4.33]{Herrero1} with the Weyl-von Neumann-Kuroda theorem fails to work for $\Vert\cdot\Vert_1$-norm. Thus, to investigate Problem A, it is necessary to develop new methods and techniques.

From the perspective mentioned above, one might encounter intrinsic difficulties when considering perturbation problems related to trace-class operators. In fact, the set of trace-class operators plays a crucial role in certain problems within operator theory, operator algebras, and scattering theory (see \cite{Han01, Simon01, Sukochev01}).

In this paper, for a large family of operators in $\mathcal{B}(\mathcal{H})$, we answer \textbf{Problem A} affirmatively.
Based on the discussion in \Cref{subsection II1-factor}, we discover that Problem A is related to the generator problem for type $\mathrm{II}_1$ factors. Thus, we propose \Cref{conj II1-factor} and prove \Cref{thm main} ({\bf Main theorem}). It is worth mentioning that while the generator problem remains unsolved for general type $\mathrm{II}_1$ factors, substantial progress has already been made toward its solution. For an overview of this problem, we refer to the book by Allan Sinclair and Roger Smith \cite[Chapter 16]{Sinclair&Smith}, as well as several papers published after this book \cite{DSSW,DSZ2015,Sherman2012,Popa2021}.

\subsection{Main theorem and an outline of the proof}

For an operator $T$  in $\mathcal{B}(\mathcal{H})$, we denote by $W^*(T)$ the von Neumann algebra generated by  $T$, by $\operatorname{Re} T$ the real part of $T$, and by $\operatorname{Im} T$ the imaginary part of $T$. A vector $\xi$ in $\mathcal{H}$ is {\em generating} or {\em cyclic} for a von Neumann algebra $\mathcal{M}$ if the set $\mathcal{M}\xi$ is dense in $\mathcal{H}$. If $\mathcal{M}$ has a cyclic vector, then $\mathcal{M}$ is said to be \emph{cyclic}.

\begin{conjecture}\label{conj II1-factor}
Suppose that $T$ is an operator in $\mathcal{B}(\mathcal{H})$ such that $W^*(T)$ is a type $\mathrm{II}_1$ factor.
Then for every $\varepsilon>0$, there exists a trace-class operator $K$ in $\mathcal{B}(\mathcal{H})$ with $\|K\|_1<\varepsilon$ such that $T+K$ is a direct sum of at most countably many irreducible operators.
\end{conjecture}

For simplicity, let $\mathrm{IR}({\mathcal{H}})$ be \emph{the set of irreducible operators} in $\mathcal{B}(\mathcal{H})$ and $\overline{\mathrm{IR}({\mathcal{H}})}^{\Vert\cdot\Vert_1}$ the closure of $\mathrm{IR}({\mathcal{H}})$ with respect to the trace-class norm topology. With Conjecture \ref{conj II1-factor}, we prove the following result in this paper.

\begin{theorem}[\textbf{Main Theorem}]\label{thm main}
    The following statements are equivalent:
    \begin{enumerate}
        \item [$(1)$]  $\overline{\mathrm{IR}({\mathcal{H}})}^{\Vert\cdot\Vert_1} = \mathcal{B}(\mathcal{H})$; 
        \item [$(2)$]  Each generator of a cyclic type $\mathrm{II}_1$ factor on $\mathcal{H}$ is in $\overline{\mathrm{IR}({\mathcal{H}})}^{\Vert\cdot\Vert_1}$; 
        \item [$(3)$]  Conjecture {\rm \ref{conj II1-factor}} is true.
    \end{enumerate}
\end{theorem}

As part of our results to show that $\overline{\mathrm{IR}({\mathcal{H}})}^{\Vert\cdot\Vert_1}$ is topologically large, each of the following subsets of $\mathcal{B}(\mathcal{H})$ is a subset of $\overline{\mathrm{IR}({\mathcal{H}})}^{\Vert\cdot\Vert_1}$:
\begin{enumerate}
    \item [$(a)$] $\{T: W^*(T) \text{ of finite type I}\}$,
    \item [$(b)$] $\{T: W^*(T) \text{ of type }\mathrm{II}_1  \text{ with nontrivial center}\}$,
    \item [$(c)$] $\{T: W^*(T) \text{ a type }\mathrm{II}_1 \text{ factor, } W^*(\operatorname{Re}T) \text{ a Cartan subalgebra} \}$,
    \item [$(d)$] $\{T: W^*(T) \text{ a factor with } W^*(\operatorname{Re}T) \text{ not diffuse}\}$,
    \item [$(e)$] $\{T: W^*(\operatorname{Re}T) \text{ a masa of } \mathcal{B}(\mathcal{H})\}$,
\end{enumerate}
where $(a)$ is from \Cref{prop-type-I},  $(b)$ is from \Cref{prop type-II},  $(c)$ is from \Cref{prop Cartan},  $(d)$ is from \Cref{cor factor CP=I}, and $(e)$  is from \Cref{cor masa}.

For the reader’s convenience, we will outline the method to prove the {\bf Main Theorem}. 
Note that $(1)\Rightarrow(2)\Rightarrow(3)$ is clear. We only need to prove  $(3)\Rightarrow(1)$. Before proceeding, we briefly recall the \emph{type decomposition theorem} for von Neumann algebras.
For a von Neumann algebra $\mathcal{M}$, by \cite[Theorem 6.5.2]{Kadison2}, there exist {\em central} projections $P_{{\rm I}_n}$ ($n\geqslant 1$), $P_{{\rm I}_\infty}$, $P_{{\rm II}_1}$, $P_{{\rm II}_\infty}$, and $P_{{\rm III}}$, with sum $I$, such that $\mathcal{M}$ can be expressed as a direct sum of von Neumann algebras in the form
\begin{equation}\label{equ type-decomposition}
  \mathcal{M}=\left(\bigoplus_{n=1}^{\infty}\mathcal{M}P_{{\rm I}_n}\right)
  \oplus\mathcal{M}P_{{\rm I}_\infty}\oplus\mathcal{M}P_{{\rm II}_1}
  \oplus\mathcal{M}P_{{\rm II}_\infty}\oplus\mathcal{M}P_{{\rm III}},
\end{equation}
where  $\mathcal{M}P_{{\rm I}_n}$ is of type ${\rm I}_n$ or $P_{{\rm I}_n}=0$, $\mathcal{M}P_{{\rm I}_\infty}$ is of type ${\rm I}_\infty$ or $P_{{\rm I}_\infty}=0$, $\mathcal{M}P_{{\rm II}_1}$ is of type ${\rm II}_1$ or $P_{{\rm II}_1}=0$, $\mathcal{M}P_{{\rm II}_\infty}$ is of type ${\rm II}_\infty$ or $P_{{\rm II}_\infty}=0$, and $\mathcal{M}P_{{\rm III}}$ is of type ${\rm III}$ or $P_{{\rm III}}=0$.
The reader is referred to \cite[Definition 6.5.1]{Kadison2} for a discussion of different types of von Neumann algebras.
For the sake of simplicity, we denote by $\mathcal{M}_{{\rm I}_{f}}$ the direct sum $\bigoplus_{n=1}^{\infty}\mathcal{M}P_{{\rm I}_n}$, which is sometimes referred to as a {\em finite type} $\mathrm{I}$ von Neumann algebra.
Also, denote by $\mathcal{M}_{\infty}$ the direct sum $\mathcal{M}P_{{\rm I}_\infty}\oplus\mathcal{M}P_{{\rm II}_\infty}\oplus\mathcal{M}P_{{\rm III}}$, which is a {\em properly infinite} von Neumann algebra (see \cite[Definition 6.3.1]{Kadison2}).
Thus, we can rewrite the decomposition in \eqref{equ type-decomposition} as
\begin{equation}\label{equ type-decomposition-2}
  \mathcal{M}=\mathcal{M}_{{\rm I}_{f}}\oplus\mathcal{M}P_{{\rm II}_1}
  \oplus\mathcal{M}_\infty.
\end{equation}

The method to prove the {\bf Main Theorem} is listed below in four steps.

{\bf Step 1.}
For an operator $T$ in $\mathcal{B}(\mathcal{H})$, we write $T=A+iB$, where $A$ and $B$ are self-adjoint operators.
By \Cref{lem Re+C1}, there exists an arbitrarily small self-adjoint trace-class operator $K_A$ such that $\Tilde{A}:=A+K_A$ and $B$ are in the form
\begin{equation}\label{equ Re+C1}
  \Tilde{A}:=
  \begin{pmatrix}
    \alpha & 0 & 0 & 0 \\
    0 & A_1 & 0 & 0 \\
    0 & 0 & A_2 & 0 \\
    0 & 0 & 0 & A_\infty
  \end{pmatrix}, \quad
  B:=
  \begin{pmatrix}
    \beta & \xi^*_1 & \xi^*_2 & \xi_\infty^* \\
    \xi_1 & B_1  & 0 & 0 \\
    \xi_2 & 0 & B_2 & 0 \\
    \xi_\infty & 0 & 0 & B_\infty
  \end{pmatrix}
  \begin{array}{l}
    \operatorname{ran}E\\
    \mathcal{H}_1 \\
    \mathcal{H}_2 \\
    \mathcal{H}_\infty \\
  \end{array}.
\end{equation}
The notation in \eqref{equ Re+C1} is explained as follows.
\begin{enumerate}
\item[$(1)$] $\alpha$ is an isolated eigenvalue of $A+K_A$ with multiplicity $1$ and $\beta\in\mathbb{R}$.

\item[$(2)$] $E$ is the spectral projection for $A+K_A$ corresponding to $\{\alpha\}$.

\item[$(3)$] $\xi_j$ is a vector in a column form and $\xi_j^*$ is the conjugate vector of $\xi_j$ in a row form for $j=1,2,\infty$.

\item[$(4)$] Let $X:=(I-E)(T+K_A)(I-E)$ be an operator on $\operatorname{ran}(I-E)$.
    According to the decomposition mentioned in \eqref{equ type-decomposition-2},  there are (mutually orthogonal) central projections $E_1$, $E_2$, and $E_\infty$ in $W^*(X)$, with sum $I-E$, such that $W^*(X)$ can be expressed as
\begin{equation}\label{equ X}
  W^*(X)=W^*(X_1)\oplus W^*(X_2)\oplus W^*(X_\infty),
\end{equation}
where $X_j:=XE_j$ for $j= 1,2,\infty$, and $W^*(X_1)$ is of finite type $\mathrm{I}$ or $E_1=0$, $W^*(X_2)$ is of type $\mathrm{II}_1$ or $E_2=0$, and $W^*(X_\infty)$ is properly infinite or $E_\infty=0$.
Correspondingly, $W^*(X_j)$ acts on $\mathcal{H}_j:=E_j\mathcal{H}$ for $j=1,2,\infty$.

\item[$(5)$] Write $A_j:=\operatorname{Re} X_j$ and $B_j:=\operatorname{Im} X_j$ for $j=1,2,\infty$.
\end{enumerate}

{\bf Step 2.}
In \eqref{equ Re+C1}, if $\mathcal{H}_1\ne 0$, then we prove in \Cref{prop-type-I} that there is an arbitrarily small trace-class operator $K_1$ in $\mathcal{B}(\mathcal{H}_1)$ such that $(A_1+iB_1)+K_1$ is  irreducible on $\mathcal{H}_1$.

{\bf Step 3.}
For each properly infinite von Neumann algebra, we prove in \Cref{lem generating-PI-VNA} that the set of generating vectors is dense. Then, we develop a method to construct irreducible operators in \Cref{lemma pre-main}, which serves for the proof of \Cref{thm main}.

{\bf Step 4.}
Assume that \Cref{conj II1-factor} is true.
Based on the above steps, we prove that $A+ iB$ can be expressed as an irreducible operator on $\mathcal{H}$ up to an arbitrarily small trace-class perturbation.

Above all, to prove the {\bf Main Theorem}, we will employ operator approximation theory with respect to the trace-class norm, single generator techniques in $\mathcal{B}(\mathcal{H})$, and methods from the theory of von Neumann algebras.

The paper is organized as follows. In Sections \ref{section-2.1} and \ref{subsection Pre-lemma}, we prepare some valuable tools. In Section \ref{subsection-2.3}, we consider single generators of properly infinite von Neumann algebras and generating vectors. \Cref{lemma pre-main} will be applied directly in the proof of \Cref{thm main}.  In \Cref{section finite-VNA}, we mainly focus on single generators of finite von Neumann algebras. In Section \ref{subsection finite}, we introduce the \emph{atomic support} for an abelian von Neumann algebra and develop a key tool in \Cref{lem relative-commutant} by \Cref{lem A-in-M}. In \Cref{subsection II1-vNalg}, we  consider the class of operators $T$ with $C_P=I$ in \Cref{lem CP=I}, where $C_P$ is the central support of the atomic support $P$ of $W^*(\operatorname{Re}T)$. In particular, if $W^*(T)$ is a factor with $W^*(\operatorname{Re}T)$ not diffuse, then $T \in \overline{\mathrm{IR}({\mathcal{H}})}^{\Vert\cdot\Vert_1}$.  Then we consider the case for $C_P<I$ in \Cref{lem CP<I}. These two lemmas yield \Cref{cor CP}, where we prove that every operator generating a diffuse finite von Neumann algebra with  nontrivial center is in $\overline{\mathrm{IR}({\mathcal{H}})}^{\Vert\cdot\Vert_1}$. As an application, in \Cref{prop-type-I}, we prove that every operator generating a finite type I von Neumann algebra is in $\overline{\mathrm{IR}({\mathcal{H}})}^{\Vert\cdot\Vert_1}$. So is every operator generating a type $\mathrm{II}_1$ von Neumann algebra with nontrivial center, by \Cref{prop type-II}. In \Cref{subsection II1-factor}, we introduce the relative normalizing set in \eqref{relative-normalizing-set} for a diffuse von Neumann subalgebra. With this concept, we develop another key tool in \Cref{lem normalizer}, which yields \Cref{prop weak-Cartan}. In \Cref{section proof}, we prove \Cref{thm main}.

\section{Preliminaries}\label{section preliminary}

\subsection{Classical tools to construct irreducible operators} \label{section-2.1}

To avoid confusion in later sections, for two vectors $e$ and $f$ in $\mathcal{H}$, we denote by $e\otimes f$ a tensor product vector in $\mathcal{H}\otimes\mathcal{H}$ and by $e\hat{\otimes} f$ we denote the rank-one operator acting on $\mathcal{H}$ defined by
\begin{equation} \label{equ rank-one}
  (e\hat{\otimes}f)(h)=\langle h,f\rangle e\quad\text{for all}~h\in\mathcal{H}.
\end{equation}
When no confusion can arise, for a vector $e$ in $\mathcal{H}$, we denote by $\|e\|:=\langle e,e\rangle^{\frac{1}{2}}$ the vector norm of $e$ and denote by $\mathrm{Tr}$ the standard trace on the set of trace-class operators.
In particular, if $e$ is a unit vector, then the rank-one operator $e\hat{\otimes}e$ is a projection and $\mathrm{Tr}(e\hat{\otimes}e)=1$.

An operator $T$ in $\mathcal{B}(\mathcal{H})$ is \emph{diagonal} if there is a family of mutually orthogonal projections $\{P_j\}^{N}_{j=1}$ with sum $I$ and a family of complex numbers $\{\lambda_j\}^{N}_{j=1}$ such that $T = \sum^{N}_{j=1} \lambda_j P_j$, where $N$ may be infinite. 
By a result of R. Carey and J. Pincus \cite[Lemma 1]{CP} and the Kato-Rosenblum theorem, a self-adjoint operator $A$ in $\mathcal{B}(\mathcal{H})$ equals its singular part if and only if for every $\varepsilon>0$ there is a self-adjoint trace-class operator $K$ with $\|K\|_1<\varepsilon$ such that $A+K$ is diagonal. For simplicity, $A$ is called \emph{purely singular} if  $A$ equals its singular part.  
Note that if $A$ is purely singular, then we can choose $K$ with $\|K\|_1<\varepsilon$ such that $A+K$ is diagonal and each eigenvalue of $A+K$ is of multiplicity $1$.
That is to say, there is an orthonormal basis $\{e_j\}_{j=1}^{\infty}$ of $\mathcal{H}$ such that
\begin{equation} \label{equ A+K}
  A+K=\sum_{j=1}^{\infty}\alpha_j e_j\hat{\otimes} e_j\quad\text{and}\quad
  \alpha_j\ne\alpha_k~\text{for all}~j\ne k.
\end{equation}

By the proofs adopted in \cite[Halmos' theorem]{Rad} and \cite[Lemma 4.33]{Herrero1}, we obtain the following result directly. For completeness, we sketch the  construction of the required trace-class operator.

\begin{lemma}\label{lem diagonalizable}
	Let $T$ be an operator in $\mathcal{B}(\mathcal{H})$ with its real part being purely singular.
	Then for every $\varepsilon>0$, there is a trace-class operator $K$ in $\mathcal{B}(\mathcal{H})$ with $\|K\|_1<\varepsilon$ such that $T+K$ is irreducible in $\mathcal{B}(\mathcal{H})$.
\end{lemma}

\begin{proof}
    Write $T = A+ iB$, where $A$ and $B$ are self-adjoint operators  in $\mathcal{B}(\mathcal{H})$. From \eqref{equ A+K}, we may assume that $A + K_1 = \sum_{j=1}^{\infty}\alpha_j e_j\hat{\otimes} e_j$ and $\alpha_j\ne\alpha_k$ for all $j\ne k$ 
    with $\|K_1\|_1<\varepsilon /2$, where $\{e_j\}_{j=1}^{\infty}$ is an orthonormal basis  of $\mathcal{H}$.
    We define a sequence $\{\delta_j\}_{j=1}^\infty$ of non-negative numbers by
    \begin{equation*}
        \delta_j=
        \begin{cases}
            0, & \mbox{if } \langle Be_{j+1},e_j\rangle\ne 0, \\
            \varepsilon, & \mbox{otherwise}.
        \end{cases}
    \end{equation*}
    Let $K_2$ be a self-adjoint trace-class operator in the form
    \begin{equation*}
        K_2 := \sum_{j=1}^{\infty}\frac{\delta_j}{2^{j+2}}
        (e_j\hat{\otimes}e_{j+1}+e_{j+1}\hat{\otimes}e_j).
    \end{equation*}
    Write $K := K_1 + i K_2$. Thus, it is routine to verify that $\|K\|_1<\varepsilon$ and $T+K$ is irreducible  in $\mathcal{B}(\mathcal{H})$.
\end{proof}

If neither the real part nor the imaginary part of $T$ is purely singular, then we need to develop some techniques in von Neumann algebras for later discussions.

\subsection{Preliminary lemmas in \texorpdfstring{$\mathcal{B}(\mathcal{H})$}{B(H)}}\label{subsection Pre-lemma}

Recall that by $\mathcal{H}$ we denote a separable infinite-dimensional complex Hilbert space.
For an operator $T$ in $\mathcal{B}(\H)$, we denote by $\operatorname{ran} T$ or $T\mathcal{H}$ the {\em range space} of $T$.

In the following \Cref{lemma-diagonal}, we prepare a routine construction.
This lemma will be directly applied in \Cref{lem P-irreducible}.
For an operator $T$ in $\mathcal{B}(\mathcal{H})$, we denote by $\sigma_p(T)$ the {\em point spectrum} of $T$, i.e., the set of all eigenvalues of $T$.
Since $\mathcal{H}$ is separable, $\sigma_p(A)$ is countable for every self-adjoint operator $A$ in $\mathcal{B}(\mathcal{H})$.
For simplicity, in a von Neumann algebra $\mathcal{M}$, a maximal abelian von Neumann subalgebra is always abbreviated as a {\em masa} in $\mathcal{M}$.

\begin{lemma} \label{lemma-diagonal}
	Let $P$ be a nonzero projection on $\mathcal{H}$, $D$ a diagonal operator on $\operatorname{ran} P$, and $\Sigma$ a countable subset of $\mathbb{R}$.
	Then for every $\varepsilon>0$, there is a self-adjoint trace-class operator $K$ in $\mathcal{B}(\mathcal{H})$ of the form
	\begin{equation*}
		K=
		\begin{pmatrix}
    		K_P & 0 \\
    		0 & 0
		\end{pmatrix}
  		\begin{array}{l}
    		\operatorname{ran} P   \\
    		\operatorname{ran} (I-P)
  		\end{array}
	\end{equation*}
	such that
	\begin{enumerate}
		\item[$(1)$] $\|K\|_1<\varepsilon$,

		\item[$(2)$] $\ker K_P=\{0\}$, i.e., $\ker K=\operatorname{ran}(I-P)$,

		\item [$(3)$] $\sigma_p(D + K_P)\cap\Sigma=\varnothing$,

		\item [$(4)$] $W^*(D+K_P)$ is a masa on $\operatorname{ran} P$ which is generated by minimal projections.
	\end{enumerate}
\end{lemma}

\begin{proof}
    Since $D$ is diagonal, there is an orthonormal basis $\{e_j\}_{j=1}^{N}$ for $\operatorname{ran} P$ such that $D$ is in the form $D=\sum_{j=1}^N\alpha_je_j\hat{\otimes}e_j$, where $N$ may be infinite.
    Choose a sequence $\{\delta_j\}_{j=1}^N$ of positive numbers such that for each $j$, we have
    \begin{enumerate}
        \item[$(1)$] $0<\delta_j<\frac{\varepsilon}{2^j}$,
        \item[$(2)$] $\alpha_j+\delta_j\notin\Sigma $,
        \item[$(3)$] $\alpha_j+\delta_j\ne\alpha_k+\delta_k$ for each $k=1,\ldots, j-1$.
    \end{enumerate}
    Define $K_P=\sum_{j=1}^N\delta_je_j\hat{\otimes}e_j$.
    Then $K_P$ is a self-adjoint trace-class operator with $\|K_P\|_1<\varepsilon$ and
    \begin{equation*}
      D+K_P=\sum_{j=1}^N(\alpha_j+\delta_j)e_j\hat{\otimes}e_j,
    \end{equation*}
    where
    \begin{enumerate}
        \item[$(1)$] $\alpha_j+\delta_j\ne\alpha_k+\delta_k $ for all $j\neq k$ and
        \item[$(2)$] $\sigma_p(D+K_P) = \{\alpha_j+\delta_j\}_{j=1}^N $. 
    \end{enumerate}
    It follows that each $e_j\hat{\otimes}e_j$ is in $W^*(D+K_P)$ by the Borel function calculus.
    Clearly, $K$ is an operator with the desired properties. 
\end{proof}

With \Cref{lemma-diagonal}, we can perturb a class of operators $A+iB$ to be irreducible with an arbitrarily small trace-class operator.

\begin{lemma} \label{lem P-irreducible}
    Let $A$ and $B$ be self-adjoint operators in $\mathcal{B}(\mathcal{H})$.
    If $W^*(A)$ contains an infinite-dimensional projection $P$ with $PB=BP$ such that $(A+iB)P$ is irreducible on $P\mathcal{H}$, then for every $\varepsilon>0$, there is a self-adjoint trace-class operator $K$ with $\|K\|_1<\varepsilon$ such that $A+i(B+K)$ is irreducible on $\mathcal{H}$.
\end{lemma}

\begin{proof}
    Let $\mathcal{H}_1=P\mathcal{H}$ and $\mathcal{H}_2=(I-P)\mathcal{H}$.
    Then $\mathcal{H}=\mathcal{H}_1\oplus\mathcal{H}_2$ and we can write
    \begin{equation*}
      A=
      \begin{pmatrix}
        A_{11} & 0 \\
        0 & A_{22}
      \end{pmatrix}\quad\text{and}\quad
      B=
      \begin{pmatrix}
        B_{11} & 0 \\
        0 & B_{22}
      \end{pmatrix}
      \begin{array}{l}
        \mathcal{H}_1 \\
        \mathcal{H}_2
      \end{array}.
    \end{equation*}
    Since $\mathcal{H}_1$ is infinite dimensional, there is a partial isometry $V$ from $\mathcal{H}_1$ onto $\mathcal{H}_2$.
    More precisely, $V$ is a partial isometry in $\mathcal{B}(\mathcal{H})$ such that
    \begin{equation*}
      V^*V\leqslant P\quad\text{and}\quad VV^*=I-P.
    \end{equation*}
    By \Cref{lemma-diagonal}, there is a self-adjoint trace-class operator $K_P$ in $\mathcal{B}(\mathcal{H}_1)$ such that $\|K_P\|_1<\frac{\varepsilon}{2}$ and $\ker K_P=0$.
    Let
    \begin{equation*}
      B_1=\begin{pmatrix}
        B_{11} & K_PV^* \\
        VK_P & B_{22}
      \end{pmatrix}.
    \end{equation*}
    Then $\|B_1-B\|_1<\varepsilon$.
    It suffices to show that $A+iB_1$ is irreducible in $\mathcal{B}(\mathcal{H})$.

    Let $Q$ be a projection commuting with $A+iB_1$.
    We will show that either $Q=0$ or $Q=I$.
    Since $Q$ commutes with $P\in W^*(A)$, $Q$ can be written as a direct sum $Q_1\oplus Q_2$, where $Q_j\in\mathcal{B}(\mathcal{H}_j)$ for $j=1,2$.
    It follows that either $Q_1=0$ or $Q_1=P$ by the irreducibility of $A_{11}+iB_{11}$.
    Without loss of generality, we assume that $Q_1=0$, otherwise we consider $I-Q$.
    Since $QB=BQ$, we have $K_PV^*Q_2=0$.
    Note that $\ker K_P=0$.
    It follows that $V^*Q_2=0$ and hence
    \begin{equation*}
      Q_2=(I-P)Q_2=VV^*Q_2=0.
    \end{equation*}
    Therefore, $Q=0$.
    This completes the proof.
\end{proof}

In the technique lemma below, if the projections $P_1,P_2$ are chosen from $W^*(A)$, then the condition $P_1,P_2\in W^*(A,B+K)$ is automatically true. 
For a subset $\mathcal{S}$ of $\mathcal{B}(\mathcal{H})$, write
\begin{equation*}
    \mathcal{S}':=\{X\in\mathcal{B}(\mathcal{H})\colon XS=SX~\text{for all}~S\in\mathcal{S}\}
\end{equation*}
to be the commutant of $\mathcal{S}$ in $\mathcal{B}(\mathcal{H})$.

\begin{lemma}\label{lem 2-projections}
    Let $A$ and $B$ be self-adjoint operators in $\mathcal{B}(\mathcal{H})$.
    Suppose that $W^*(B)'$ contains two infinite-dimensional projections $P_1$ and $P_2$ with sum $I$ such that
    \begin{equation*}
        P_1,P_2\in W^*(A,B+K)
    \end{equation*}
    for every self-adjoint compact operator $K$ in $\mathcal{B}(\mathcal{H})$.
    Then for every $\varepsilon>0$, there is a self-adjoint trace-class operator $K\in\mathcal{B}(\mathcal{H})$ with $\|K\|_1<\varepsilon$ such that $A+i(B+K)$ is irreducible in $\mathcal{B}(\mathcal{H})$.
\end{lemma}

\begin{proof}
    Let $\mathcal{H}_j=P_j\mathcal{H}$ for $j=1,2$.
    Then we can write $\mathcal{H}=\mathcal{H}_1\oplus\mathcal{H}_2$ and
    \begin{equation*}
      B=
      \begin{pmatrix}
        B_{11} & 0 \\
        0 & B_{22}
      \end{pmatrix}
      \begin{array}{l}
        \mathcal{H}_1 \\
        \mathcal{H}_2
      \end{array}.
    \end{equation*}
    Let $\{e_j\}_{j=1}^\infty$ and $\{f_j\}_{j=1}^\infty$ be orthonormal bases for $\mathcal{H}_1$ and $\mathcal{H}_2$, respectively.
    As in the proof of \Cref{lem diagonalizable}, we define a sequence $\{\delta_j\}_{j=1}^\infty$ of non-negative numbers by
    \begin{equation*}
      \delta_j=
      \begin{cases}
        0, & \mbox{if } \langle B_{11}e_{j+1},e_j\rangle\ne 0, \\
        \varepsilon, & \mbox{otherwise}.
      \end{cases}
    \end{equation*}
    Let
    \begin{equation*}
      K_1=\sum_{j=1}^{\infty}\frac{\delta_j}{2^{j+2}}
      (e_j\hat{\otimes}e_{j+1}+e_{j+1}\hat{\otimes}e_j)
      \quad\text{and}\quad
      K_2=\sum_{j=1}^{\infty}\frac{\varepsilon}{2^{j+2}}f_j\hat{\otimes}e_j.
    \end{equation*}
    It is clear that $\|K_1\|_1\leqslant\frac{\varepsilon}{2}$ and $\|K_2\|_1\leqslant\frac{\varepsilon}{4}$.
    Moreover, we have
    \begin{equation}\label{equ B11}
      \langle(B_{11}+K_1)e_{j+1},e_j\rangle\ne 0\quad\text{for all}~j\geqslant 1.
    \end{equation}
    We define a self-adjoint operator $B_1$ in $\mathcal{B}(\mathcal{H})$ by
    \begin{equation*}
      B_1=
      \begin{pmatrix}
        B_{11}+K_1 & K_2^* \\
        K_2 & B_{22}
      \end{pmatrix}
      \begin{array}{l}
        \mathcal{H}_1 \\
        \mathcal{H}_2
      \end{array}.
    \end{equation*}
    Then $\|B_1-B\|_1 \leqslant\varepsilon$.
    It suffices to show that $A+iB_1$ is irreducible in $\mathcal{B}(\mathcal{H})$.

    By assumption, we have $P_1, P_2\in W^*(A+iB_1)$.
    Then $K_2=P_2B_1P_1$ belongs to $W^*(A+iB_1)$.
    It follows that
    \begin{equation*}
      K_2^*K_2=\sum_{j=1}^{\infty}\frac{\varepsilon^2}{4^{j+2}}e_j\hat{\otimes}e_j
      \in W^*(A+iB_1).
    \end{equation*}
    By means of the Borel function calculus for the positive operator $K_2^*K_2$, we obtain that
    \begin{equation*}
      e_j\hat{\otimes}e_j\in W^*(A+iB_1).
    \end{equation*}
    By considering the operator $(e_j\hat{\otimes}e_j)B_1(e_{j+1}\hat{\otimes}e_{j+1})$, it follows from \eqref{equ B11} that
    \begin{equation}\label{equ e-e}
      e_j\hat{\otimes}e_{j+1}\in W^*(A+iB_1).
    \end{equation}
    Since $K_2(e_j\hat{\otimes}e_j)\in W^*(A+iB_1)$, we see that
    \begin{equation}\label{equ f-e}
      f_j\hat{\otimes}e_j\in W^*(A+iB_1).
    \end{equation}
    Note that $\{e_j\hat{\otimes}e_{j+1}\}_{j=1}^\infty$ and $\{f_j\hat{\otimes}e_j\}_{j=1}^\infty$ generate $\mathcal{B}(\mathcal{H})$ as a von Neumann algebra.
    Therefore, $A+iB_1$ is irreducible in $\mathcal{B}(\mathcal{H})$ by \eqref{equ e-e} and \eqref{equ f-e}.
    This completes the proof.
\end{proof}

The following consequence of \Cref{lem 2-projections} states that if the projection $I-P$ in \Cref{lem P-irreducible} is also infinite dimensional, then we can remove the condition $(A+iB)P$ being irreducible on $P\mathcal{H}$.

\begin{corollary}\label{cor 2-projections}
	Let $A$ and $B$ be self-adjoint operators in $\mathcal{B}(\mathcal{H})$.
	If $W^*(A)$ contains two infinite-dimensional projections $P_1$ and $P_2$ such that
	\begin{equation*}
		P_1+P_2=I\quad\text{and}\quad P_jB=BP_j\quad\text{ for }~j = 1,2,
	\end{equation*}
	then for every $\varepsilon>0$, there is a self-adjoint trace-class operator $K\in\mathcal{B}(\mathcal{H})$ with $\|K\|_1<\varepsilon$ such that $A+i(B+K)$ is irreducible in $\mathcal{B}(\mathcal{H})$.
\end{corollary}

To reveal a tip of the efficiency of \Cref{cor 2-projections}, we provide a short proof of Theorem 4.1 of \cite{Shi3}.

\begin{corollary}\label{cor normal}
	For each normal operator $N$ in $\mathcal{B}(\mathcal{H})$ and $\varepsilon>0$, there is a trace-class operator $K$ in $\mathcal{B}(\mathcal{H})$ with $\|K\|_1<\varepsilon$ such that $N+K$ is irreducible in $\mathcal{B}(\mathcal{H})$.
\end{corollary}

\begin{proof}
	Write $N=A+iB$, where $A$ and $B$ are self-adjoint operators in $\mathcal{B}(\mathcal{H})$.
	If $W^*(A)$ is finite dimensional, then $A$ is diagonal. We finish the proof by applying \Cref{lem diagonalizable}.

	If $W^*(A)$ is infinite dimensional, then there is a sequence $\{E_n\}_{n=1}^\infty$ of nonzero projections in $W^*(A)$ with sum $I$.
	Define a projection $P$ in the form $P:= \sum_{n=1}^{\infty}E_{2n}$.
	It follows that $P$ and $I-P$ are both infinite-dimensional projections in $\mathcal{B}(\mathcal{H})$.
	Thus, the proof is completed by applying \Cref{cor 2-projections}.
\end{proof}

Note that $A+iB$ is normal if and only if $W^*(A)\subseteq W^*(A+iB)'$.
Thus, it is natural to consider {\bf Problem A} for operators $A+iB$ satisfying the reverse inclusion $W^*(A+iB)'\subseteq W^*(A)$.
Before proceeding to the following \Cref{prop reverse}, we make an observation in \Cref{rem reverse}.

\begin{remark}\label{rem reverse}
	For any self-adjoint operators $A$ and $B$ in $\B(\H)$, it is obvious to have the inclusion $W^*(A+iB)'\subseteq W^*(A)'$.
	Moreover, assume that $W^*(A)$ is a masa of $\mathcal{B}(\mathcal{H})$, which is equivalent to the inclusion $W^*(A)'\subseteq W^*(A)$.
	The two inclusions imply that
	\begin{equation}\label{equ reverse}
  		W^*(A+iB)' \subseteq W^*(A).
	\end{equation}
	Generally speaking, besides the set of operators $A+iB$ with $W^*(A)$ a masa, there is also a large family of operators satisfying \eqref{equ reverse}, such as irreducible operators.
\end{remark}

As an application of \Cref{cor 2-projections}, we obtain the following proposition.

\begin{proposition}\label{prop reverse}
	Let $A$ and $B$ be self-adjoint operators in $\mathcal{B}(\mathcal{H})$ such that
	\begin{equation*}
	W^*(A+iB)'\subseteq W^*(A).
	\end{equation*}
	Then for every $\varepsilon>0$, there exists a self-adjoint trace-class operator $K$ with $\|K\|_1<\varepsilon$ such that $A+i(B+K)$ is irreducible in $\mathcal{B}(\mathcal{H})$.
\end{proposition}

\begin{proof}
	Since $W^*(A)$ is an abelian von Neumann algebra, the hypothesis entails that $W^*(A+iB)'$ is also an abelian von Neumann algebra.

	If $W^*(A+iB)'$ is finite dimensional, then there is an infinite-dimensional minimal projection $P$ in $W^*(A+iB)'$.
	It follows that $(A+iB)P$ is irreducible on $P\mathcal{H}$.
	Thus, by applying \Cref{lem P-irreducible}, there exists a self-adjoint trace-class operator $K$ in $\mathcal{B}(\mathcal{H})$ with $\|K\|_1<\varepsilon$ such that $A+i(B+K)$ is irreducible.

	If $W^*(A+iB)'$ is infinite dimensional, then there exists a sequence $\{E_n\}_{n=1}^\infty$ of nonzero projections in $W^*(A+iB)'$ such that $I= \sum_{n=1}^{\infty}E_n$.
	Write $P:= \sum_{n=1}^{\infty}E_{2n}$.
	It follows that both $P$ and $I-P$ are infinite-dimensional projections in $\mathcal{B}(\mathcal{H})$.
	Thus by applying \Cref{cor 2-projections}, there exists a self-adjoint trace-class operator $K$ in $\mathcal{B}(\mathcal{H})$ with $\|K\|_1<\varepsilon$ such that $A+i(B+K)$ is irreducible.
    
	The above two cases complete the proof.
\end{proof}

By \Cref{prop reverse}, we have a direct corollary.

\begin{corollary}\label{cor masa}
	Let $A$ and $B$ be self-adjoint operators in $\mathcal{B}(\mathcal{H})$ such that $W^*(A)$ is a masa of $\mathcal{B}(\mathcal{H})$.
	Then for every $\varepsilon>0$, there exists an irreducible operator $Y$ in $\mathcal{B}(\mathcal{H})$ such that
	\begin{equation*}  
		\|(A+iB)-Y\|_1<\varepsilon.
	\end{equation*}
\end{corollary}

\begin{remark}\label{rem masa}
	One may think that if for each self-adjoint operator $A$ in $\mathcal{B}(\mathcal{H})$ there exists an arbitrarily small self-adjoint trace-class operator $K$ such that $W^*(A+K)$ is a masa of $\mathcal{B}(\mathcal{H})$, then {\bf Problem A} can be solved completely by applying Corollary $\ref{cor masa}$.
	But the thought fails to work.
	We provide such a self-adjoint operator without a proof.
	Let $M_t$ be the multiplication operator on $L^2[0,1]$ defined by $(M_tf)(t):=t\cdot f(t)$ for every $f\in L^2[0,1]$.
	Let $\mathcal{H}=L^2[0,1]\oplus L^2[0,1]$ and $A:= M_t\oplus M_t$.
	Clearly, $W^*(A)$ is not a masa in $\mathcal{B}(\mathcal{H})$.
	By Theorem $5.2.5$ of \cite{Li2}, for each self-adjoint trace-class operator $K$, $W^*(A+K)$ fails to be a masa in $\mathcal{B}(\mathcal{H})$.
\end{remark}

\subsection{Cyclic vectors for properly infinite von Neumann algebras} \label{subsection-2.3}

Recall that a von Neumann algebra $\mathcal{M}$ is said to be \emph{properly infinite} if the identity operator $I$ is properly infinite in $\mathcal{M}$, which is equivalent to saying that each central projection in $\mathcal{M}$ is either infinite or zero.
The reader is referred to \cite[Definition 6.3.1]{Kadison2} for more details.
For two projections $P$ and $Q$ in $\mathcal{M}$, if there exists a partial isometry $V$ in $\mathcal{M}$ such that $V^*V=P$ and $VV^* = Q$, then $P$ and $Q$ are said to be {\em Murray-von Neumann equivalent} and we denote by $P\sim Q$ this equivalence relation.

\begin{lemma}\label{lem PI-VNA}
	Let $\mathcal{M}$ be a properly infinite von Neumann algebra.
	Then there is a system of matrix units $\{E_{jk}\}_{j,k=1}^{\infty}$ in $\mathcal{M}$ such that $\sum_{j=1}^{\infty}E_{jj}=I$.
\end{lemma}

\begin{proof}
	By \cite[Lemma 6.3.3]{Kadison2}, there are projections $P_1,Q_1$ in $\mathcal{M}$ such that $I=P_1+Q_1$ and $P_1\sim Q_1\sim I$.
	Similarly, there are projections $P_2,Q_2$ such that $Q_1=P_2+Q_2$ and $P_2\sim Q_2\sim I$.
	Inductively, we can define $P_n,Q_n$ such that $Q_{n-1}=P_n+Q_n$ and $P_n\sim Q_n\sim I$.
	Let
	\begin{equation*}
		E_1=P_1+(I-P_2-P_3-\cdots),\quad E_2=P_2,\quad E_3=P_3,\quad \ldots.
	\end{equation*}
	Then $I=\sum_{j=1}^{\infty}E_j$ and $E_j\sim I$ for each $j$.
	Let $E_{j1}$ be a partial isometry such that $E_{j1}^*E_{j1}=E_1$ and $E_{j1}E_{j1}^*=E_j$.
	We define $E_{ij}:=E_{i1}E_{1j}$.
	Then $\{E_{jk}\}_{j,k=1}^{\infty}$ is a system of matrix units in $\mathcal{M}$ such that $\sum_{j=1}^{\infty}E_{jj}=I$.
\end{proof}

It is worth mentioning that, in Exercise VIII.1 $(8)$ of \cite{Take2}, if the set of generating vectors for a von Neumann algebra $\mathcal{M}$ is non-empty, then it is a dense $G_\delta$-set in $\mathcal{H}$.
To perturb an operator to be irreducible in the trace-class norm, we develop the following characterization of properly infinite von Neumann algebras with respect to generating vectors.

\begin{lemma}\label{lem generating-PI-VNA}
    Let $\mathcal{M}$ be a properly infinite von Neumann algebra acting on $\mathcal{H}$.
    Then the set of generating vectors of $\mathcal{M}$ is dense in $\mathcal{H}$.
\end{lemma}

\begin{proof}
    By applying \Cref{lem PI-VNA}, there is a system of matrix units $\{E_{jk}\}_{j,k=1}^{\infty}$ in $\mathcal{M}$ such that $\sum_{j=1}^{\infty}E_{jj}=I$.
    Let $\mathcal{N}=E_{11}\mathcal{M}E_{11}\subseteq\mathcal{B}(E_{11}\mathcal{H})$. We define a unitary operator $U\colon\mathcal{H}\to\ell^2\otimes E_{11}\mathcal{H}$ by
    \begin{equation*}
      U\xi=\sum_{j=1}^{\infty}e_j\otimes E_{1j}\xi,
    \end{equation*}
    where $\{e_j\}_{j=1}^{\infty}$ is an orthonormal basis for $\ell^2$.
    We write $F_{jk}:= e_j\hat{\otimes}e_k$ for all $j,k\geqslant 1$.
    For any vectors $\xi$ and $\eta$ in $\mathcal{H}$, we have
    \begin{align*}
      \langle U^*(F_{jk}\otimes I_{\mathcal{N}})U\xi,\eta\rangle
      & =\langle(F_{jk}\otimes I_{\mathcal{N}})\sum_{\ell=1}^{\infty}e_\ell\otimes E_{1\ell}\xi,\sum_{\ell =1}^{\infty}e_\ell\otimes E_{1\ell}\eta\rangle \\
      & =\langle (F_{jk} \otimes I_{\mathcal{N}})(e_k\otimes E_{1k}\xi),e_j\otimes E_{1j}\eta\rangle \\
      & =\langle e_j\otimes E_{1k}\xi,e_j\otimes E_{1j}\eta\rangle
      =\langle E_{1k}\xi,E_{1j}\eta\rangle
      =\langle E_{jk}\xi,\eta\rangle.
    \end{align*}
    Then $\{F_{jk}\}_{j,k=1}^{\infty}$ is a system of matrix units in $\mathcal{B}(\ell^2)$ satisfying
    \begin{equation*}
      UE_{jk}U^*=F_{jk}\otimes I_{\mathcal{N}}\quad\text{for all}~j,k\geqslant 1.
    \end{equation*}
    It is routine to verify that  $U\mathcal{M}U^* = \mathcal{B}(\ell^2)\,\overline{\otimes}\,\mathcal{N}$.
    Without loss of generality, we assume that
    \begin{equation*}
      \mathcal{M}=\mathcal{B}(\ell^2)\,\overline{\otimes}\,\mathcal{N},\quad \mathcal{H}=\ell^2\otimes\mathcal{H}_0,\quad \mathcal{N}\subseteq\mathcal{B}(\mathcal{H}_0).
    \end{equation*}
    Let $\xi=\sum_{j=1}^{\infty} e_j\otimes\xi_j \in \ell^2\otimes\mathcal{H}_0$ and $\varepsilon>0$.
    Then there is a sufficiently large integer $n$ such that $\sum_{j=n+1}^{\infty}\|\xi_j\|^2 < \frac{\varepsilon^2}{4}$.
    Let $\{f_j\}_{j=1}^{\infty}$ be an orthonormal basis for $\mathcal{H}_0$ and construct a vector $\eta$ in the form
    \begin{equation*}
      \eta := \sum_{j=1}^{n}e_j\otimes\xi_j
           +  \sum_{k=1}^{\infty}\frac{\varepsilon}{2^{k+1}}e_{n+k}\otimes f_k.
    \end{equation*}
    Then $\|\xi-\eta\| < \varepsilon$.
    Moreover, for every $j,k\geqslant 1$, we have
    \begin{equation*}
      e_j\otimes f_k=\frac{2^{k+1}}{\varepsilon}(F_{j,n+k}\otimes I_{\mathcal{N}})\eta \in \mathcal{M}\eta.
    \end{equation*}
    Thus, $\eta$ is a generating vector for $\mathcal{M}$.    
\end{proof}

According to \Cref{lem generating-PI-VNA}, there exist numerous generating vectors for a properly infinite von Neumann algebra $\mathcal{M}$ acting on $\mathcal{H}$.
The following proposition is a related application about generating vectors.

\begin{proposition} 
    Let $\mathcal{M}$ be a von Neumann algebra acting on $\mathcal{H}$ with a generating vector $\xi$.
    Then
    \begin{equation*}
    	W^*(\mathcal{M},\xi\hat{\otimes}\xi)=\mathcal{B}(\mathcal{H}).
    \end{equation*}
\end{proposition}

\begin{proof}
    Note that for every $T_1$ and $T_2$ in $\mathcal{M}$, we have
    \begin{equation*}
      T_1\xi\hat{\otimes}T_2\xi=T_1(\xi\hat{\otimes}\xi)T_2^*\in W^*(\mathcal{M},\xi\hat{\otimes}\xi).
    \end{equation*}
    Since $\xi$ is a generating vector for $\mathcal{M}$, the set $\mathcal{M}\xi$ is dense in $\mathcal{H}$.
    It follows that the weak-operator closure of $\mathrm{span}\{T_1\xi\hat{\otimes}T_2\xi\colon T_1,T_2\in\mathcal{M}\}$ equals $\mathcal{B}(\mathcal{H})$.
    This completes the proof.
\end{proof}

By applying \Cref{lemma-diagonal}, we prove the following result, which plays an essential role in the proof of \Cref{thm main}.

\begin{proposition}\label{lemma pre-main}
    Suppose that $A$ and $B$ are self-adjoint operators in $\mathcal{B}(\mathcal{H})$  of the form
    \begin{equation*}
        A:=
        \begin{pmatrix}
        A_{11} & 0 \\
        0 & A_{22}
        \end{pmatrix}
        \quad\text{and}\quad
        B:=
        \begin{pmatrix}
        B_{11} & B_{12} \\
        B^*_{12} & B_{22}
        \end{pmatrix}
        \begin{array}{l}
        \mathcal{H}_1\\
        \mathcal{H}_2
        \end{array},
    \end{equation*}
    where $\mathcal{H}=\mathcal{H}_1\oplus\mathcal{H}_2$, $\mathcal{H}_1\ne\{0\}$, and
    \begin{enumerate}
        \item[$(1)$] $A_{11}$ is a diagonal operator on $\mathcal{H}_1$,
        
        \item[$(2)$] the set of generating vectors for $W^*(A_{22}+i B_{22})$ is dense in $\mathcal{H}_2$.
    \end{enumerate}
    Then for every $\varepsilon>0$, there exists a trace-class operator $K$ in $\mathcal{B}(\mathcal{H})$ with $\|K\|_1<\varepsilon$ such that the operator $(A+iB)+K$ is irreducible in $\mathcal{B}(\mathcal{H})$.
\end{proposition}

\begin{proof}
    Since $A_{11}$ is diagonal, by \Cref{lemma-diagonal}, there exists a self-adjoint trace-class operator $K_1$ in $\mathcal{B}(\mathcal{H}_1)$ with $\| K_1 \|_1 < \frac{\varepsilon}{4}$ such that $A_{11}+K_1$ is a diagonal operator with distinct eigenvalues and $\sigma_p(A_{11}+K_1) \cap \sigma_p(A_{22})=\varnothing$.
    Similar to the proof of \Cref{lem diagonalizable}, there is a self-adjoint operator $K_2$ in $\mathcal{B}(\mathcal{H}_1)$ with $\|K_2\|_1 < \frac{\varepsilon}{4}$ such that $(A_{11}+K_1)+i(B_{11}+K_2)$ is irreducible in $\mathcal{B}(\mathcal{H}_1)$.

    Given a unit vector $\eta$ in $\H_1$, by the hypothesis that the set of generating vectors for $W^*(A_{22}+i B_{22})$ is dense in $\mathcal{H}_2$, there is a vector $\xi$ in $\mathcal{H}_2$ with $\|\xi\|<\frac{\varepsilon}{4}$ such that $B_{21}\eta+\xi$ is a generating vector for $W^*(A_{22}+iB_{22})$.
    Thus, $B_{21}\eta+\xi$ is a separating vector for $W^*(A_{22}+iB_{22})'$.
    Let
    \begin{equation*}
      A_1=
      \begin{pmatrix}
        A_{11}+K_1 & 0 \\
        0 & A_2
      \end{pmatrix}\quad\text{and}\quad
      B_1=
      \begin{pmatrix}
        B_{11}+K_2 & B_{12}+\eta\hat{\otimes}\xi  \\
        B_{21}+\xi\hat{\otimes}\eta & B_{22}
      \end{pmatrix}.
    \end{equation*}
    Then $\|(A_1+iB_1)-(A+iB)\|_1<\varepsilon$.
    It suffices to prove that $A_1+iB_1$ is irreducible in $\mathcal{B}(\mathcal{H})$.

    Since $A_{11}+K_1$ is diagonal and $\sigma_p(A_{11}+K_1)\cap\sigma_p(A_{22})=\varnothing$, we have that $I_1\oplus 0\in W^*(A_1+iB_1)$.
    It follows that
    \begin{equation*}
      (A_{11}+K_1)\oplus 0,(B_{11}+K_2)\oplus 0\in W^*(A_1+iB_1).
    \end{equation*}
    Thus, $\mathcal{B}(\mathcal{H}_1)\oplus 0\subseteq W^*(A_1+iB_1)$.

    Let $Q$ be a projection commuting with $A_1+iB_1$.
    Then $Q$ can be written as $Q_1\oplus Q_2$, and we have either $Q_1=0$ or $Q_1=I_1$.
    Without loss of generality, we assume that $Q_1=0$.
    Since $QB_1=B_1Q$, we obtain that $Q_2(B_{21}+\xi\hat{\otimes}\eta)=0$.
    It follows that
    \begin{equation*}
      Q_2(B_{21}\eta+\xi)=Q_2(B_{21}+\xi\hat{\otimes}\eta)\eta=0.
    \end{equation*}
    Note that $Q_2\in W^*(A_2+iB_{22})'$ and $B_{21}\eta+\xi$ is a separating vector for $W^*(A_2+iB_{22})'$.
    Therefore, we have $Q_2=0$ and $Q=0$.
    This completes the proof.
\end{proof}

We present a remark related to finite von Neumann algebras.

\begin{remark}\label{rem direct-sum-finite-VNA}
    Let $\{T_{\lambda}\}_{\lambda\in\Lambda}$ be a family of operators such that each $W^*(T_{\lambda})$ is a finite von Neumann algebra acting on $\mathcal{H}_{\lambda}$.
    Clearly, $\bigoplus_{\lambda\in\Lambda} W^*(T_{\lambda})$ is a finite von Neumann algebra by applying Lemma $6.3.6$ of \cite{Kadison2}.
    Note that $W^*(\bigoplus_{\lambda\in\Lambda} T_{\lambda})$ is a von Neumann subalgebra of $\bigoplus_{\lambda\in\Lambda} W^*(T_{\lambda})$.
    Employing Proposition $6.3.2$ of \cite{Kadison2}, $W^*(\bigoplus_{\lambda\in\Lambda}T_{\lambda})$ is a finite von Neumann algebra acting on $\bigoplus_{\lambda\in\Lambda}\mathcal{H}_{\lambda}$.
\end{remark}


In view of \Cref{rem direct-sum-finite-VNA}, we prove an analogous result for operators $\{T_{\lambda}\}_{\lambda\in\Lambda}$, where each $W^*(T_{\lambda})$ is a properly infinite von Neumann algebra.

\begin{lemma}\label{lem direct-sum-PI-VNA}
    Let $\{T_{\lambda}\}_{\lambda\in\Lambda}$ be a family of operators such that each $W^*(T_{\lambda})$ is a properly infinite von Neumann algebra acting on $\mathcal{H}_{\lambda}$. Then $W^*(\bigoplus_{\lambda\in\Lambda} T_{\lambda})$ is a properly infinite von Neumann algebra acting on $\bigoplus_{\lambda\in\Lambda} \mathcal{H}_{\lambda}$.
\end{lemma}

\begin{proof}
    Without loss of generality, each $\mathcal{H}_{\lambda}$ is  viewed as a subspace of $\mathcal{H}=\bigoplus_{\lambda\in\Lambda}\mathcal{H}_{\lambda}$.
    Let $E_{\lambda}'$ be the projection from $\mathcal{H}$ onto $\mathcal{H}_{\lambda}$ for all $\lambda\in\Lambda$.
    
    Write $T:=\bigoplus_{\lambda\in\Lambda} T_{\lambda}$.
    It is clear that $\{E_{\lambda}'\}_{\lambda\in\Lambda}$ is a family of projections in $W^*(T)'$ with sum $I$.
    Let $P$ be a finite central projection in $W^*(T)$.
    We prove that $P=0$ as follows.
    By applying Proposition 6.3.2 of \cite{Kadison2}, we obtain that $PC_{E_{\lambda}'}$ is a finite central projection in $W^*(TC_{E_{\lambda}'})$, where $C_{E_{\lambda}'}$ is the central support of $E_{\lambda}'$ in $W^*(T)^{\prime}$ for each $\lambda \in \Lambda$.
    By Proposition 5.5.5 of \cite{Kadison1}, $PE_{\lambda}'$ is a finite central projection in $W^*(TE_{\lambda}')=W^*(T_{\lambda})$.
    Since each $W^*(T_{\lambda})$ is properly infinite, we see that $PE_{\lambda}'=0$ for every $\lambda \in \Lambda$.
    It follows that $P=0$.
    This completes the proof.
\end{proof}


    

\section{Perturbation of single generators of finite von Neumann algebras} \label{section finite-VNA}

Let $T$ be an operator in $\mathcal{B}(\mathcal{H})$. Here are two brief applications we obtain in this section.
If $W^*(T)$ is a finite type $\mathrm{I}$ von Neumann algebra, then $T \in \overline{\mathrm{IR}({\mathcal{H}})}^{\Vert\cdot\Vert_1}$ 
(see \Cref{prop-type-I}).
If $W^*(T)$ is a type $\mathrm{II}_1$ von Neumann algebra with nontrivial center, then the same conclusion holds for $T$ (see \Cref{prop type-II}).

\subsection{Finite von Neumann algebras} \label{subsection finite}

Recall that the pair $(\mathcal{M},\tau)$ is called a {\em tracial von Neumann algebra} if $\mathcal{M}$ is a finite von Neumann algebra and $\tau$ is a normal faithful tracial state on $\mathcal{M}$.
By the GNS construction, the normal faithful tracial state $\tau$ induces a normal *-isomorphism $\pi_\tau$ from $\mathcal{M}$ onto the von Neumann algebra $\pi_{\tau}(\mathcal{M})$ acting on $L^2(\mathcal{M},\tau)$.
Since $\tau$ is faithful, every operator $X$ in $\mathcal{M}$ can be viewed as a vector $\widehat{X}$ in $L^2(\mathcal{M},\tau)$ and the inner product on $\widehat{\mathcal{M}}$ (as a dense subset of $L^2(\mathcal{M},\tau)$) is defined by
\begin{equation*}
	\langle\widehat{X},\widehat{Y}\rangle=\tau(Y^*X)\quad\text{for all}~X,Y\in\mathcal{M}.
\end{equation*}
In particular, we write $\hat{1}$ as the vector in $L^2(\mathcal{M},\tau)$ corresponding to the identity operator $I$ in $\mathcal{M}$.
Thus, for every $X\in\mathcal{M}$, we also write the vector $\widehat{X}$ as $X\hat{1}$.
By definition, we have
\begin{equation*}
	\pi_\tau(T)\widehat{X} = \widehat{TX}\quad\text{and}\quad
	\tau(T) = \langle\pi_\tau(T)\hat{1},\hat{1}\rangle\quad\text{for all}~ T,X\in\mathcal{M}.
\end{equation*}
Moreover, $\pi_\tau$ is called the {\em standard representation} of $(\mathcal{M},\tau)$.
When no confusion can arise, we will omit $\pi_\tau$ for each operator $T$ in $\mathcal{M}$ acting on $L^2(\mathcal{M},\tau)$.

Recall that a unitary operator $U$ in a tracial von Neumann algebra $(\mathcal{M},\tau)$ is said to be a {\em Haar unitary operator} if $\tau(U^n)=0$ for all $n\in\mathbb{Z}\backslash\{0\}$.

\begin{lemma}\label{lem Haar-power}
	Let $\mathcal{A}$ be a diffuse abelian von Neumann algebra acting on $\mathcal{H}$.
	Then there exists a sequence $\{U_n\}_{n=1}^\infty$ of unitary operators in $\mathcal{A}$ weak-operator convergent to $0$.
\end{lemma}

\begin{proof}
	Let $\tau$ be a normal faithful tracial state on $\mathcal{A}$.
	By \cite[Theorem 3.5.2]{Sinclair&Smith}, there is a *-isomorphism $\theta$ from $\mathcal{A}$ onto $L^\infty[0,1]$ such that
	\begin{equation*}
		\tau(A) = \int_{0}^{1}\theta(A)(t)dt\quad\text{for all}~A\in\mathcal{A}.
	\end{equation*}
	Let $u(t)=e^{2\pi it}$ and $U=\theta^{-1}(u)$.
	Since $u$ is a Haar unitary operator in $L^\infty[0,1]$, $U$ is a Haar unitary operator in $(\mathcal{A},\tau)$.
	We claim that the sequence $\{U^n\}_{n=1}^{\infty}$ is weak-operator convergent to $0$.

	According to the GNS construction, the normal faithful tracial state $\tau$ induces a normal *-isomorphism $\pi_{\tau}$ from $\mathcal{A}$ onto the von Neumann algebra $\pi_{\tau}(\mathcal{A})$ acting on $L^2(\mathcal{A},\tau)$.
	By applying \cite[Corollary 7.1.16]{Kadison2}, it suffices to prove that $\{\pi_{\tau}(U^n)\}_{n=1}^\infty$ is weak-operator convergent to $0$.
	Since $U$ is a Haar unitary operator in $\mathcal{A}$, the sequence $\{\widehat{U^n}\}_{n=1}^\infty$ is an orthonormal subset of $L^2(\mathcal{A},\tau)$.
	Thus, for any vectors $\widehat{X}$ and $\widehat{Y}$ in $\widehat{\mathcal{A}}$ (as a dense subset of $L^2(\mathcal{A},\tau)$), we obtain that
	\begin{equation*}
		\langle\pi_\tau(U^n)\widehat{X},\widehat{Y}\rangle
  		=\tau(U^nXY^*)=\langle\widehat{U^n},\widehat{YX^*}\rangle\to 0\quad\text{as}~n\to\infty.
	\end{equation*}
	Thus, $\{\pi_{\tau}(U^n)\}_{n=1}^\infty$ is weak-operator convergent to $0$.
	This completes the proof.
\end{proof}

The following lemma is well known to experts.

\begin{lemma}\label{lem KUn}
    Suppose that $\{U_n\}_{n=1}^\infty$ is a sequence of unitary operators weak-operator convergent to $0$ in $\mathcal{B}(\mathcal{H})$.
    Then for every compact operator $K$ in $\mathcal{B}(\mathcal{H})$, the sequence $\{KU_n\}_{n=1}^\infty$ is strong-operator convergent to $0$.
\end{lemma}

\begin{proof}
    For any vector $\xi\in\mathcal{H}$, the sequence $\{U_n\xi\}^{\infty}_{n = 1}$ is weakly convergent to $0$ in $\mathcal{H}$.
    Since $K$ is compact, we have $\|KU_n\xi\|\to 0$ as $n\to\infty$.
    This completes the proof.
\end{proof}

Recall that a set $\mathcal{D}$ is called a {\em total subset} of $\mathcal{H}$ if the closed linear span of $\mathcal{D}$ equals $\mathcal{H}$.
The following lemma provides a criterion for determining whether a self-adjoint operator in $\mathcal{B}(\mathcal{H})$ belongs to $\mathcal{M}$.

\begin{lemma}\label{lem A-in-M}
    Let $\mathcal{M}$ be a von Neumann algebra acting on $\mathcal{H}$ and $\mathcal{D}$ a total subset of $\mathcal{H}$. Suppose that $A$ in $\mathcal{B}(\mathcal{H})$ is self-adjoint for which, given any $\xi\in\mathcal{D}$ and $\varepsilon>0$, there is a self-adjoint operator $B$ in $\mathcal{M}$ such that $\|(A-B)\xi\|<\varepsilon$.
    Then $A\in\mathcal{M}$.
\end{lemma}

\begin{proof}
    Let $E'$ be a projection in $\mathcal{M}'$, $\xi$ a vector in $\mathcal{D}$, and $\varepsilon>0$.
    By assumption, there is a self-adjoint operator $B$ in $\mathcal{M}$ such that the inequality $\|(A-B)\xi\|\cdot\|\xi\|<\frac{\varepsilon}{2}$ holds.
    It follows that
    \begin{align*}
        |\langle(AE'-E'A)\xi,\xi\rangle|
        &=|\langle((A-B)E'-E'(A-B))\xi,\xi\rangle| \\
        &\leqslant|\langle E'\xi,(A-B)\xi\rangle|+|\langle(A-B)\xi,E'\xi\rangle|\\
        &\leqslant 2\|(A-B)\xi\|\cdot\|\xi\|<\varepsilon.
    \end{align*}
    Thus, $\langle(AE'-E'A)\xi,\xi\rangle=0$ for each vector $\xi\in\mathcal{D}$.
    By the polarization identity and continuity of the inner product, we obtain  $AE'=E'A$ for every projection $E'\in\mathcal{M}'$.
    Therefore, $A\in\mathcal{M}$. This completes the proof.
\end{proof}

The following corollary is immediate.

\begin{corollary}\label{cor A-in-M}
    Let $\mathcal{M}$ be a von Neumann algebra acting on $\mathcal{H}$ and $\mathcal{D}$ a total subset of $\mathcal{H}$. Suppose that $T$ is in $\mathcal{B}(\mathcal{H})$ for which, given any $\xi\in\mathcal{D}$ and $\varepsilon>0$, there is an operator $S$ in $\mathcal{M}$ such that $\|(T-S)\xi\|<\varepsilon$ and $\|(T-S)^*\xi\|<\varepsilon$.
	Then $T\in\mathcal{M}$.
\end{corollary}

If $(\mathcal{M},\tau)$ is a tracial von Neumann algebra, then we define
\begin{equation*}
  \|X\|_{2,\tau}=\tau(X^*X)^{1/2}\quad\text{for all}~X\in\mathcal{M}.
\end{equation*}
For an abelian von Neumann algebra $\mathcal{A}$, its {\em atomic support} is defined as the sum of all minimal projections in $\mathcal{A}$.
If $A$ is a self-adjoint operator and $P$ is the atomic support of $W^*(A)$, then clearly $AP=PA$ is a diagonal operator.
The next lemma is related to \Cref{lem 2-projections}.

\begin{lemma}\label{lem relative-commutant}
    Let  $A$ and $B$ be self-adjoint operators in $\mathcal{B}(\mathcal{H})$. Suppose that $\mathcal{M} = W^*(A+iB)$ is a finite von Neumann algebra and $P$ is the atomic support of $W^*(A)$. Then for every self-adjoint compact operator $K$ on $\mathcal{H}$, we have
    \begin{equation*}
    	W^*(A)'\cap(I-P)\mathcal{M}(I-P)\subseteq W^*(A,B+K).
    \end{equation*}
\end{lemma}

\begin{proof}
	Since $\mathcal{M}$ is a finite von Neumann algebra, there is a normal faithful tracial state $\tau$ on $\mathcal{M}$.
	Without loss of generality, we assume that
	\begin{equation*}
		\mathcal{H}=\bigoplus_{j=1}^\infty L^2(\mathcal{M}P_j,\tau),
	\end{equation*}
	where $\{P_j\}_{j=1}^\infty$ is a sequence of projections in $\mathcal{M}$ and some $P_j$'s may be $0$.
	For each $j\geqslant 1$, let $\mathcal{D}_j$ be the linear manifold of all vectors $\xi=\{\xi_k\}_{k=1}^\infty$ satisfying that $\xi_j\in\mathcal{M}\widehat{P_j}$ and $\xi_k=0$ for each $k\neq j$.
	Write $\mathcal{D}=\bigcup_{j=1}^\infty\mathcal{D}_j$.
	Then $\mathcal{D}$ is a total subset of $\mathcal{H}$.
	Let $\xi$ be a unit vector in $\mathcal{D}$ and $\varepsilon>0$.
	We may assume that $\xi=X\widehat{P_j}\oplus 0\in\mathcal{D}_j$ for some nonzero $X\in\mathcal{M}$ and $j\geqslant 1$.

	Let $T$ be an operator in $W^*(A)'\cap(I-P)\mathcal{M}(I-P)$.
	By the Kaplansky density theorem (Theorem 5.3.5 of \cite{Kadison1}), there is a non-commutative polynomial $p$ such that
	\begin{equation}\label{equ 2-tau}
		\|T-p(A,B)\|_{2,\tau}=\|(T-p(A,B))^*\|_{2,\tau}<\frac{\varepsilon}{2\|X\|}.
	\end{equation}
	Since $K$ is compact, we can write $p(A,B+K)=p(A,B)+K_0$, where $K_0$ is compact.

	Since $W^*(A)(I-P)$ is a diffuse abelian von Neumann algebra acting on $(I-P)\mathcal{H}$, by \Cref{lem Haar-power}, there is a sequence of partial isometries $\{U_n\}^{\infty}_{n=1}$ in $W^*(A)$ weak-operator convergent to $0$ and $U_n^*U_n=U_nU_n^*=I-P$.
	By \Cref{lem KUn}, there exists an integer $N\geqslant 1$ such that for all $n\geqslant N$, we have
	\begin{equation}\label{equ KUn}
		\|K_0U_n\xi\|<\frac{\varepsilon}{2}\quad\text{and}\quad \|K_0^*U_n\xi\|<\frac{\varepsilon}{2}.
	\end{equation}
	Since $TU_n=U_nT$, it follows from \eqref{equ 2-tau} and \eqref{equ KUn} that
	\begin{align*}
  		\|(T-U_n^*p(A,B+K)U_n)\xi\|
  		&\leqslant\|(T-U_n^*p(A,B)U_n)\xi\|+\|U_n^*K_0U_n\xi\|\\
  		&=\|U_n^*(T-p(A,B))U_nXP_j\|_{2,\tau}+\|K_0U_n\xi\|\\
  		&\leqslant\|T-p(A,B)\|_{2,\tau}\|X\|+\frac{\varepsilon}{2}<\varepsilon.
	\end{align*}
	Similarly, we have $\|(T-U_n^*p(A,B+K)U_n)^*\xi\|<\varepsilon$.
	Therefore, by \Cref{cor A-in-M}, we can obtain that $T\in W^*(A,B+K)$.
	This completes the proof.
\end{proof}

\subsection{\texorpdfstring{$\Vert \cdot \Vert_1$}{trace-class}-norm perturbation of operators in finite von Neumann algebras} \label{subsection II1-vNalg}

In this subsection, we show that a large class of operators can be perturbed into irreducible operators.
Recall that the atomic support of an abelian von Neumann algebra $\mathcal{A}$ is the sum of all minimal projections in $\mathcal{A}$.

\begin{lemma}\label{lem CP=I}
    Let $A$ and $B$ be self-adjoint operators in $\mathcal{B}(\mathcal{H})$. 
    Suppose that $W^*(A)Z$ is not diffuse for every nonzero central projection $Z$ in $W^*(A+iB)$. Then for every $\varepsilon>0$, there is a trace-class operator $K\in\mathcal{B}(\mathcal{H})$ with $\|K\|_1<\varepsilon$ such that $(A+iB)+K$ is irreducible in $\mathcal{B}(\mathcal{H})$.
\end{lemma}

\begin{proof}
    Let $P$ be the atomic support of $W^*(A)$.
    Then for every central projection $Z$ in $W^*(A+iB)$ with $ZP=0$, $W^*(A)Z$ is diffuse or $Z=0$.
    By assumption, we have $Z=0$.
    Therefore, $C_P=I$, where $C_P$ is the central support of $P$ in $W^*(A+iB)$.
    Let $\mathcal{H}_1=P\mathcal{H}$ and $\mathcal{H}_2=(I-P)\mathcal{H}$.
    Then we can write $\mathcal{H}=\mathcal{H}_1\oplus\mathcal{H}_2$ and
    \begin{equation*}
      A=
      \begin{pmatrix}
        A_{11} & 0 \\
        0 & A_{22}
      \end{pmatrix}\quad\text{and}\quad
      B=
      \begin{pmatrix}
        B_{11} & B_{12} \\
        B_{21} & B_{22}
      \end{pmatrix}
      \begin{array}{l}
        \mathcal{H}_1\\
        \mathcal{H}_2
      \end{array}.
    \end{equation*}
    Since $P$ is the atomic support of $W^*(A)$, $A_{11}$ is a diagonal self-adjoint operator on $\mathcal{H}_1$.
    According to \Cref{lemma-diagonal}, there exists a self-adjoint trace-class operator $K_1$ in $\mathcal{B}(\mathcal{H}_1)$ with $\|K_1\|_1 < \frac{\varepsilon}{2}$ such that $A_{11}+K_1$ is a diagonal operator with distinct eigenvalues and $\sigma_p(A_{11}+K_1)\cap\sigma_p(A_{22})=\varnothing$.
    By applying \Cref{lem diagonalizable}, there exists a self-adjoint operator $K_2$ in $\mathcal{B}(\mathcal{H}_1)$ with $\|K_2\|_1 < \frac{\varepsilon}{2}$ such that $(A_{11}+K_1)+i(B_{11}+K_2)$ is irreducible in $\mathcal{B}(\mathcal{H}_1)$.
    Let $A_1$ and $B_1$ be self-adjoint operators of the form
    \begin{equation*}
      A_1:=
      \begin{pmatrix}
        A_{11}+K_1 & 0 \\
        0 & A_{22}
      \end{pmatrix}\quad\text{and}\quad
      B_1:=
      \begin{pmatrix}
        B_{11}+K_2 & B_{12} \\
        B_{21} & B_{22}
      \end{pmatrix}.
    \end{equation*}
    It suffices to prove that $A_1+iB_1$ is irreducible in $\mathcal{B}(\mathcal{H})$.
    
    Since $A_{11}+K_1$ is diagonal and $\sigma_p(A_{11}+K_1)\cap\sigma_p(A_{22})=\varnothing$, we can obtain that $P\in W^*(A_1+iB_1)$.
    It follows that
    \begin{equation*}
      (A_{11}+K_1)\oplus 0,(B_{11}+K_2)\oplus 0\in W^*(A_1+iB_1).
    \end{equation*}
    Thus, $\mathcal{B}(\mathcal{H}_1)\oplus 0\subseteq W^*(A_1+iB_1)$.
    In particular, $A,B\in W^*(A_1+iB_1)$.

    Let $Q$ be a projection commuting with $A_1+iB_1$.
    Since $QP=PQ$, $Q$ can be written as $Q = Q_1\oplus Q_2$, where $Q_j\in\mathcal{B}(\mathcal{H}_j)$ for $j=1,2$.
    Then either $Q_1=0$ or $Q_1=P$.
    Without loss of generality, we can assume that $Q_1=0$. From $Q$ commuting with $A_1+iB_1$ and $\{A,B\}\subseteq W^*(A_1+iB_1)$,
    we have $QXP=XPQ=0$ for every $X\in W^*(A+iB)$.
    It follows that $QC_P=0$, i.e., $Q=0$.
    This completes the proof.
\end{proof}

As a direct application of \Cref{lem CP=I}, we obtain the following corollary.

\begin{corollary}\label{cor factor CP=I}
    Let $A$ and $B$ be self-adjoint operators in $\mathcal{B}(\mathcal{H})$ such that $W^*(A+iB)$ is a factor.
    If $W^*(A)$ is not diffuse, then for every $\varepsilon>0$, there is a trace-class operator $K\in\mathcal{B}(\mathcal{H})$ with $\|K\|_1<\varepsilon$ such that $(A+iB)+K$ is irreducible in $\mathcal{B}(\mathcal{H})$.
\end{corollary}

Compared to \Cref{lem CP=I}, we have the following result. In the remaining part of this subsection, for self-adjoint operators  $A$ and $B$ in $\mathcal{B}(\mathcal{H})$, we assume that the von Neumann algebra $W^*(A+iB)$ is finite.

\begin{lemma}\label{lem CP<I}
    Let $A$ and $B$ be self-adjoint operators in $\mathcal{B}(\mathcal{H})$ such that $W^*(A+iB)$ is a finite von Neumann algebra.
    Suppose that $Z$ is a central projection in $W^*(A+iB)$ such that $W^*(A)Z$ is diffuse and $I-Z$ is infinite dimensional.
    Then for every $\varepsilon>0$, there is a trace-class operator $K\in\mathcal{B}(\mathcal{H})$ with $\|K\|_1<\varepsilon$ such that $(A+iB)+K$ is irreducible in $\mathcal{B}(\mathcal{H})$.
\end{lemma}

\begin{proof}
    Since $W^*(A)Z$ is diffuse, $Z$ is infinite dimensional.
    Let $P$ be the atomic support of $W^*(A)$.
    Then $Z\leqslant I-P$.
    It follows from \Cref{lem relative-commutant} that $Z\in W^*(A,B+K)$ for every self-adjoint compact operator $K$ on $\mathcal{H}$.
    The remaining part of the proof is completed by \Cref{lem 2-projections}.
\end{proof}

By \Cref{lem CP=I} and \Cref{lem CP<I}, we obtain the following proposition.

\begin{proposition}\label{cor CP}
    Let $A$ and $B$ be self-adjoint operators in $\mathcal{B}(\mathcal{H})$ such that $W^*(A+iB)$ is a diffuse finite von Neumann algebra and is not a factor.
    Then for every $\varepsilon>0$, there is a trace-class operator $K\in\mathcal{B}(\mathcal{H})$ with $\|K\|_1<\varepsilon$ such that $(A+iB)+K$ is irreducible in $\mathcal{B}(\mathcal{H})$.
\end{proposition}

\begin{proof}
    If $W^*(A)Z$ is not diffuse for every nonzero central projection $Z$ in $W^*(A+iB)$, then we complete the proof by \Cref{lem CP=I}.
    Thus, we assume that there exists a central projection $Z_0$ in $W^*(A+iB)$ such that $W^*(A)Z_0$ is diffuse.
    
    Note that $W^*(A+iB)$ is not a factor by assumption.
    If $Z_0=I$, then there exists a central projection $Z$ in $W^*(A+iB)$ with $0<Z<I$.
    Clearly, $W^*(A)Z$ is diffuse.
    If $0<Z_0<I$, then we choose $Z=Z_0$.
    Since $W^*(A+iB)$ is diffuse, $I-Z$ is infinite dimensional.
    Therefore, we complete the proof by \Cref{lem CP<I}.
\end{proof}

The following proposition is a special case of \Cref{cor CP}.

\begin{proposition}\label{prop type-II}
    Let $A$ and $B$ be self-adjoint operators in $\mathcal{B}(\mathcal{H})$. Suppose that $W^*(A+iB)$ is a type $\mathrm{II}_1$ von Neumann algebra and is not a factor.
    Then for every $\varepsilon>0$, there is a trace-class operator $K\in\mathcal{B}(\mathcal{H})$ with $\|K\|_1<\varepsilon$ such that $(A+iB)+K$ is irreducible in $\mathcal{B}(\mathcal{H})$.
\end{proposition}

Let $P$ be the atomic support of $W^*(A)$.
The following lemma shows that, in a finite von Neumann algebra, the condition $C_P=I$ in \Cref{lem CP=I} can be replaced by the condition $P$ being infinite dimensional.

\begin{lemma}\label{lem P-infinite}
    Let $A$ and $B$ be self-adjoint operators in $\mathcal{B}(\mathcal{H})$
    such that $W^*(A+iB)$ is a finite von Neumann algebra and the atomic support of $W^*(A)$ is an infinite-dimensional projection.
    Then for every $\varepsilon>0$, there exists a trace-class operator $K\in\mathcal{B}(\mathcal{H})$ with $\|K\|_1<\varepsilon$ such that $(A+iB)+K$ is irreducible in $\mathcal{B}(\mathcal{H})$.
\end{lemma}

\begin{proof}
    Let $P$ be the atomic support of $W^*(A)$ and $C_P$ the central projection of $P$ in $W^*(A+iB)$.
    If $C_P=I$, then we complete the proof by \Cref{lem CP=I}.
    Assume that $0<C_P<I$.
    Since $P$ is infinite dimensional, $C_P$ is also infinite dimensional.
    Moreover, $W^*(A)(I-C_P)$ is diffuse.
    Therefore, we complete the proof by \Cref{lem CP<I}.
\end{proof}

The following corollary is a direct application of \Cref{lem P-infinite}.

\begin{corollary}\label{cor P-infinite}
    Let $A$ and $B$ be self-adjoint operators in $\mathcal{B}(\mathcal{H})$ such that $W^*(A+iB)$ is a diffuse finite von Neumann algebra and $W^*(A)$ is not diffuse.
    Then for every $\varepsilon>0$, there is a trace-class operator $K\in\mathcal{B}(\mathcal{H})$ with $\|K\|_1<\varepsilon$ such that $(A+iB)+K$ is irreducible in $\mathcal{B}(\mathcal{H})$.
\end{corollary}

The following proposition is an enhanced version of \Cref{cor normal}.

\begin{proposition} \label{prop-type-I}
    Let $A$ and $B$ be self-adjoint operators in $\mathcal{B}(\mathcal{H})$. Suppose that $W^*(A+iB)$ is a finite type $\mathrm{I}$ von Neumann algebra.
    Then for every $\varepsilon>0$, there is a trace-class operator $K\in\mathcal{B}(\mathcal{H})$ with $\|K\|_1<\varepsilon$ such that $(A+iB)+K$ is irreducible in $\mathcal{B}(\mathcal{H})$.
\end{proposition}

\begin{proof}
    Following the proof of \Cref{cor CP}, we may assume that there exists a central projection $Z_0$ in $W^*(A+iB)$ such that $W^*(A)Z_0$ is diffuse.
    Since $W^*(A+iB)Z_0$ is also a finite type $\mathrm{I}$ von Neumann algebra, there is a central projection in $W^*(A+iB)$ such that $0<Z<Z_0$.
    Clearly, $W^*(A)Z$ and $W^*(A)(Z_0-Z)$ are diffuse.
    In particular, $I-Z$ is infinite dimensional.
    The proof is completed by \Cref{lem CP<I}.
\end{proof}

\begin{remark}
    Among examples of single generators $A+iB$ of type $\mathrm{II}_1$ factors, for each separable type $\mathrm{II}_1$ factor $\M$ with $\G(\M)=0$, by Theorem $16.4.5$ of \cite{Sinclair&Smith}, it is not hard to construct two self-adjoint operators $A$ and $B$ with $W^*(A+iB)=\M$ and $W^*(A)$ not diffuse. It is also worth mentioning that for each type $\mathrm{II}_1$  factor $\mathcal{N}$ associated with groups $SL_n(\mathbb{Z})$ for $n\geqslant 3$, by Theorem $5$ of \cite{Ge1}, there exists two self-adjoint operators $A$ and $B$ with $W^*(A+iB)=\N$ and $W^*(A)$ not diffuse. Note that these  type $\mathrm{II}_1$  factors $\mathcal{N}$ have property $(T)$. An interesting question is whether $\mathcal{L}(\mathbf{F}_2)$ has a single generator $A+iB$ with either $W^*(A)$ or $W^*(B)$ not diffuse.
\end{remark}

\subsection{\texorpdfstring{$\Vert \cdot \Vert_1$}{trace-class}-norm perturbation of operators in type \texorpdfstring{$\mathrm{II}_1$}{II1} factors} \label{subsection II1-factor}

Let $A$ and $B$ be self-adjoint operators in $\mathcal{B}(\mathcal{H})$ such that $W^*(A+iB)$ is a finite von Neumann algebra.
By \Cref{prop type-II} and \Cref{prop-type-I}, we only need to consider the case $W^*(A+iB)$ being a type $\mathrm{II}_1$ factor.
Moreover, we may assume that $W^*(A)$ is diffuse by \Cref{cor P-infinite}.

Let $(\mathcal{M},\tau)$ be a type $\mathrm{II}_1$ factor with a unique normal faithful tracial state $\tau$.
Given a von Neumann subalgebra $\mathcal{A}$ of $\mathcal{M}$, the \emph{normalizer} $N_{\mathcal{M}}(\mathcal{A})$ of  $\mathcal{A}$ in $\mathcal{M}$ is of the form
\begin{equation*}  \label{normalizer}
    N_{\mathcal{M}}(\mathcal{A}) := \{V\in\mathcal{M}\colon V\mathcal{A}V^* = \mathcal{A},~ V~\text{unitary}\}.
\end{equation*}
Moreover, $\mathcal{A}$  is said to be a \emph{Cartan subalgebra} of  $\mathcal{M}$ if $N_{\mathcal{M}}(\mathcal{A})^{\prime\prime} = \mathcal{M}$.
If $\mathcal{A}$ and $\mathcal{B}$ are abelian von Neumann algebras such that $\mathcal{A}\subseteq\mathcal{B}\subseteq\mathcal{M}$, then
the \emph{relative normalizing set} $RN_{\mathcal{M}}(\mathcal{A},\mathcal{B})$ is defined as
\begin{equation}  \label{relative-normalizing-set}
    RN_{\mathcal{M}}(\mathcal{A},\mathcal{B}) := \{V\in\mathcal{M}\colon V\mathcal{A}V^*\subseteq\mathcal{B},~ V~\text{unitary}\}.
\end{equation}
We shall write $RN_{\mathcal{M}}(\mathcal{A})$ for $RN_{\mathcal{M}}(\mathcal{A},\mathcal{A})$. It is evident that $N_{\mathcal{M}}(\mathcal{A})\subseteq RN_{\mathcal{M}}(\mathcal{A})$ and they are not always equal.
Compared to \Cref{lem relative-commutant}, we have the following result.

\begin{lemma}\label{lem normalizer}
    Let $A$ and $B$ be self-adjoint operators in $\mathcal{B}(\mathcal{H})$. Suppose that $(\mathcal{M},\tau) = W^*(A+iB)$ is a type $\mathrm{II}_1$ factor and $\mathcal{A}$ is a diffuse von Neumann subalgebra of $W^*(A)$.
    Then for any self-adjoint compact operator $K$ in $\mathcal{B}(\mathcal{H})$, we have
    \begin{equation*}
        RN_{\mathcal{M}}(\mathcal{A},W^*(A))\subseteq W^*(A,B+K).
    \end{equation*}
\end{lemma}

\begin{proof}
    We adopt the same notation $\mathcal{D}_j$ and $\mathcal{D}$ as in the proof of \Cref{lem relative-commutant} with $\mathcal{D}=\bigcup_{j=1}^\infty\mathcal{D}_j$ and $\mathcal{D}$ a total subset of $\mathcal{H}$. Let $\xi$ be a unit vector in $\mathcal{D}$ and $\varepsilon>0$. Assume that $\xi=X\widehat{P_j}\oplus 0\in\mathcal{D}_j$ for some nonzero $X\in\mathcal{M}$ and $j\geqslant 1$.
    
    Let $V$ be a unitary operator in $RN_{\mathcal{M}}(\mathcal{A},W^*(A))$.
    By the Kaplansky density theorem, 
    there exists a non-commutative polynomial $p$ such that
    \begin{equation}\label{equ 2-tau-V}
        \|V-p(A,B)\|_{2,\tau}=\|(V-p(A,B))^*\|_{2,\tau}<\frac{\varepsilon}{2\|X\|}.
    \end{equation}
    Since $K$ is compact, we can write $p(A,B+K)=p(A,B)+K_0$, where $K_0$ is compact.
    
    Since $\mathcal{A}$ is diffuse, by \Cref{lem Haar-power}, there is a sequence $\{U_n\}_{n=1}^\infty$ of unitary operators in $\mathcal{A}$ weak-operator convergent to $0$.
    Let $W_n=VU_nV^*\in W^*(A)$.
    Then $\{W_n\}_{n=1}^\infty$ is also weak-operator convergent to $0$.
    By \Cref{lem KUn}, there exists an integer $N\geqslant 1$ such that for all $n\geqslant N$, we have
    \begin{equation}\label{equ-KUn}
        \|K_0U_n\xi\|<\frac{\varepsilon}{2}\quad\text{and}\quad \|K_0^*W_n\xi\|<\frac{\varepsilon}{2}.
    \end{equation}
    Note that $V=W_n^*VU_n$.
    It follows from \eqref{equ 2-tau-V} and \eqref{equ-KUn} that
    \begin{align*}
        \|(V-W_n^*p(A,B+K)U_n)\xi\|
        &\leqslant\|(V-W_n^*p(A,B)U_n)\xi\|+\|W_n^*K_0U_n\xi\|\\
        &=\|W_n^*(V-p(A,B))U_nXP_j\|_{2,\tau}+\|K_0U_n\xi\|\\
        &\leqslant\|V-p(A,B)\|_{2,\tau}\|X\|+\frac{\varepsilon}{2}<\varepsilon.
    \end{align*}
    Similarly, we have $\|(V-W_n^*p(A,B+K)U_n)^*\xi\|<\varepsilon$.
    Therefore, by \Cref{cor A-in-M}, we can obtain that $V\in W^*(A,B+K)$.
    This completes the proof.        
\end{proof}

\begin{proposition}\label{prop weak-Cartan}
    Let $A$ and $B$ be self-adjoint operators in $\mathcal{B}(\mathcal{H})$. Suppose that $\mathcal{M} = W^*(A+iB)$  is a type $\mathrm{II}_1$ factor and $\mathcal{A}$ is a diffuse von Neumann subalgebra of $W^*(A)$ such that $W^*\big({RN}_{\mathcal{M}}(\mathcal{A}, W^*(A) )\big) \cap W^*(B)^{\prime} \neq \mathbb{C}I.$
    
    Then for every $\varepsilon>0$, there exists a self-adjoint trace-class operator $K$ with $\|K\|_1<\varepsilon$ such that $A+i(B+K)$ is irreducible  in $\mathcal{B}(\mathcal{H})$.
\end{proposition}

\begin{proof}
    Let $P$ be a nontrivial projection in $W^*({RN}_{\mathcal{M}}(\mathcal{A}, W^*(A) ))  \cap W^*(B)^{\prime}$. By \Cref{lem normalizer}, for every self-adjoint compact operator $K$ in $\mathcal{B}(\mathcal{H})$, $P$ is in $W^*(A,B+K)$. Since $\mathcal{M}$ is  a type $\mathrm{II}_1$ factor, both $P$ and $I-P$ are infinite-dimensional projections in $\mathcal{B}(\mathcal{H})$. Thus, the proof can be finished by \Cref{lem 2-projections}.
\end{proof}

By applying \Cref{prop weak-Cartan}, we obtain the following proposition.

\begin{proposition}\label{prop Cartan}
    Let $A$ and $B$ be self-adjoint operators in $\mathcal{B}(\mathcal{H})$. Suppose that $\mathcal{M}=W^*(A+iB)$ is a type $\mathrm{II}_1$ factor  and $W^*(A)$ is a Cartan subalgebra of $\mathcal{M}$.
    Then for every $\varepsilon>0$, there is a self-adjoint trace-class operator $K$ with $\|K\|_1<\varepsilon$ such that $A+i(B+K)$ is  irreducible  in $\mathcal{B}(\mathcal{H})$.
\end{proposition}

Inspired by \Cref{prop Cartan}, we propose \Cref{conj II1-factor}.
To provide further support for \Cref{conj II1-factor}, we present, in the following example, an operator that generates a non-hyperfinite type $\mathrm{II}_1$ factor.

In the following example, by $\mathcal{L}(\mathbf{F}_r)$ we denote the \emph{interpolated free group factor} for $1<r\leqslant\infty$ (see \cite[Postscript]{Voi92}). By applying \cite[Corollary 5.7]{DSSW}, $\mathcal{L}(\mathbf{F}_r)$ can be generated by $\lceil r\rceil$ self-adjoint operators for $1<r\leqslant\infty$. Recall that for any real number $x$, let $\lfloor x\rfloor$ and $\lceil x\rceil$ be the floor function and the ceiling function of $x$, respectively.

\begin{example} \label{example-free-group}
    Let $\mathcal{L}(\mathbf{F}_p)$ act on a complex separable Hilbert space $\mathcal{H}$ for $1<p<2$. Then there are self-adjoint operators $A$ and $B$ generating $\mathcal{L}(\mathbf{F}_p)$ such that for any $\varepsilon>0$, there exists a self-adjoint trace-class operator $K$ with $\|K\|_1 \leqslant \varepsilon$ such that $A+i(B+K)$ is  a direct sum of irreducible operators.
\end{example}

\begin{proof}
    Let $n\geqslant 3$ be an integer such that $\lceil 1+n^2(p-1)\rceil\leqslant n^2$. Then $\mathcal{L}(\mathbf{F}_{1+n^2(p-1)})$ can be generated by $n^2$ positive operators
    \begin{equation*}
        A_j\ (1\leqslant j\leqslant n); \ B_{jk}\ (1\leqslant j\leqslant k\leqslant n, (j,k)\ne(1,1)); \ C_{jk} \ (3\leqslant j+2\leqslant k\leqslant n).
    \end{equation*}
    Furthermore, we can assume that $\sigma(A_j)\cap\sigma(A_k)=\varnothing$ for $j\ne k$, and each $B_{j,j+1}$ is invertible for $1\leqslant j \leqslant n-1$.
    For $1\leqslant j\leqslant k\leqslant n$, define 
    \begin{equation*}
        X_{jk}:=
        \begin{cases}
            0, & \text{for } (j,k) = (1,1), \\
            B_{jk}+iC_{jk}, & \text{for } 3\leqslant j+2\leqslant k\leqslant n, \\ 
            B_{jk}, & \mbox{otherwise}.
        \end{cases}
    \end{equation*}
    Let $X_{jk}=X_{kj}^*$ for $1\leqslant k< j \leqslant n$.
    We define
    \begin{equation*}
        A=\mathrm{diag}(A_1,A_2,\ldots,A_n),\quad B=(X_{jk})_{1\leqslant j,k\leqslant n}.
    \end{equation*}
    Then it is routine to verify that $W^*(A+iB) = \mathcal{L}(\mathbf{F}_p)={\mathbb M}^{}_{n}(\mathbb{C})\otimes\mathcal{L}(\mathbf{F}_{1+n^2(p-1)})$ (see \cite[Theorem 5.1]{Voi2016}). By considering the commutant of $\mathcal{L}(\mathbf{F}_p)$, we assume that $\mathcal{L}(\mathbf{F}_p)$ acts on $\mathbb{C}^n\otimes\mathcal{H}_0$, where $\mathcal{L}(\mathbf{F}_{1+n^2(p-1)})$ acts on $\mathcal{H}_0$ and has a generating unit vector $\xi$. 
    
    Write each operator of $\mathcal{L}(\mathbf{F}_p)$ as an $n\times n$ matrix with entries in the factor $\mathcal{L}(\mathbf{F}_{1+n^2(p-1)})$, and define a self-adjoint trace-class operator $K$ to be the  operator-valued  matrix with the $(1,1)$-entry being $ \varepsilon \xi\hat{\otimes}\xi$ on $\H_0$ and $0$ at all other entries. That is to say, we can write $K$ in the form
    \begin{equation*}
        K:=
        \begin{pmatrix}
            \varepsilon \xi\hat{\otimes}\xi & 0 & \cdots & 0 \\
            0 & 0 & \cdots & 0  \\
            \vdots & \vdots & \ddots & \vdots  \\
            0 & 0 & \cdots & 0 
        \end{pmatrix}.
    \end{equation*}
    Note that $(\xi,0,\ldots,0)$ is the generating vector for  $\mathcal{L}(\mathbf{F}_p)$.  A routine calculation ensures that $A+i(B+K)$ is irreducible on  $\mathbb{C}^n\otimes\mathcal{H}_0$.
\end{proof}

With the method adopted in \Cref{example-free-group}, we can construct a family of examples to support \Cref{conj II1-factor}.

\section{Proof of the Main Theorem}\label{section proof}

With these technical tools developed in the preceding sections, we are ready to prove \Cref{thm main} ({\bf Main Theorem}) in this section.

For every self-adjoint operator $A$, we denote by $E_A(\,\cdot\,)$ the {\em spectral measure} for $A$.
In the following lemma, we recall a classical technique of constructing an arbitrarily small self-adjoint trace-class perturbation $K$ of $A$ such that $\sigma(A+K)$ contains an isolated eigenvalue of multiplicity $1$.

\begin{lemma}\label{lem Re+C1}
	Let $A$ be a {self-adjoint} operator in $\mathcal{B}(\mathcal{H})$.
	Then for every $\varepsilon >0$, there exists a  {self-adjoint} finite rank operator $K$ such that
	\begin{enumerate}
		\item[$(1)$] $\|K\|_1<\varepsilon$, and

		\item [$(2)$] there is an isolated eigenvalue of $A+K$ with multiplicity $1$.
	\end{enumerate}
\end{lemma}

\begin{proof}
	Without loss of generality, we assume that $A$ is positive and $\|A\|=1$.
	In this case, we have $1\in\sigma(A)$.
	Define a Borel set $\Delta_\varepsilon$ of the form
	\begin{equation*}
  		\Delta_\varepsilon:=\sigma(A)\cap[1- \textstyle\frac{\varepsilon}{4},1].
	\end{equation*}
	By assumption, $E_A(\Delta_{\varepsilon})$ is a nonzero spectral projection for $A$ and
	\begin{equation*}
  		\|(A-I)E_A(\Delta_{\varepsilon})\|\leqslant\frac{\varepsilon}{4}.
	\end{equation*}
	Choose a unit vector $\xi$  in $\operatorname{ran} E_A(\Delta_\varepsilon)$ and denote by $F_\xi$  the rank-one projection $\xi\hat{\otimes}\xi$ onto $\mathbb{C}\xi$. Then we have
	\begin{align*}
  		\|F_\xi(A-I)F_\xi\|_1
  		\leqslant\|(A-I)F_\xi\|_1
  		&=\|(A-I)E_A(\Delta_\varepsilon)F_\xi\|_1 \\
  		&\leqslant\|(A-I)E_A(\Delta_\varepsilon)\|\|F_\xi\|_1
 	 	\leqslant\frac{\varepsilon}{4}.
	\end{align*}
	Since $F_\xi(A-I)=((A-I)F_\xi)^*$, we have $\|F_\xi(A-I)\|_1=\|(A-I)F_\xi\|_1$.
	Define a finite rank operator $K$ by
	\begin{equation*}
	    K:=\frac{\varepsilon}{4}F_\xi+F_\xi(A-I)F_\xi-F_\xi(A-I)-(A-I)F_\xi.
	\end{equation*}
	Then $\|K\|_1\leqslant\varepsilon$ and $A+K$ is in the form
	\begin{equation}\label{equ spectrum-FAF}
	    A+K=(1+\frac{\varepsilon}{4})F_\xi+(I-F_\xi)A(I-F_\xi).
	\end{equation}
	Note that $\|(I-F_\xi)A(I-F_\xi)\|\leqslant\|A\|=1$.
	From the restriction of $(I-F_\xi)A(I-F_\xi)$ on $\operatorname{ran}(I-F_\xi)$, we have
	\begin{equation}\label{equ spectrum-FPAFP}
	    \sigma((I - F_\xi)A(I - F_\xi))\subseteq [0,1].
	\end{equation}
	By \eqref{equ spectrum-FAF} and \eqref{equ spectrum-FPAFP},  $1+\frac{\varepsilon}{4}$ is an isolated eigenvalue of $A+K$ with multiplicity $1$.
	This completes the proof.
\end{proof}

\begin{proof}[\bf Proof of \Cref{thm main}]
	Let $T$ be an operator in $\mathcal{B}(\mathcal{H})$ and $\varepsilon>0$. It is sufficient to prove $(3) \Rightarrow (1)$.
	The proof is divided into 4 steps.

	{\bf Step 1.}
	By \Cref{lem Re+C1} and the type decomposition theorem for von Neumann algebras, there is a self-adjoint trace-class operator $K_0$ satisfying $\|K_0\|_1<\frac{\varepsilon}{4}$ such that $A:=\operatorname{Re}T+K_0$ and $B:=\operatorname{Im} T$ are of the form
	\begin{equation*}
  	A=
  	\begin{pmatrix}
	    \alpha & 0  & 0  & 0 \\
	    0 & A_1  & 0 & 0 \\
	    0 & 0 & A_2  & 0 \\
	    0 & 0 & 0 & A_\infty \\
  	\end{pmatrix} \quad \text{ and } \quad
  	B=
  	\begin{pmatrix}
	    \beta & \xi^*_0  & \xi^*_1  & \xi^*_2 \\
	    \xi_0 & B_1  & 0 & 0 \\
		\xi_1 & 0 & B_2  & 0 \\
	    \xi_2 & 0 & 0 & B_\infty \\
  	\end{pmatrix}
  	\begin{array}{l}
	    \operatorname{ran} E\\
	    \mathcal{H}_1 \\
	    \mathcal{H}_2 \\
	    \mathcal{H}_\infty \\
  	\end{array},
	\end{equation*}
	where
	\begin{enumerate}
	\item [$(1)$] $\alpha$ is an isolated eigenvalue of $A$ with multiplicity $1$, and $E$ is the rank-one spectral projection of $A$ corresponding to $\{\alpha\}$;

	\item [$(2)$] $W^*(A_1+iB_1)$ is a finite type $\mathrm{I}$ von Neumann algebra in $\mathcal{B}(\mathcal{H}_1)$;

	\item [$(3)$] $W^*(A_2+iB_2)$ is a type $\mathrm{II}_1$ von Neumann algebra in $\mathcal{B}(\mathcal{H}_2)$;

	\item [$(4)$]  $W^*(A_\infty+iB_\infty)$ is a properly infinite von Neumann algebra in $\mathcal{B}(\mathcal{H}_\infty)$.
	\end{enumerate}

	{\bf Step 2.}
	If $\mathcal{H}_1=\{0\}$, then let $K_1=0$.
	Otherwise, by \Cref{prop-type-I}, there is a trace-class operator $K_1$ in $\mathcal{B}(\mathcal{H}_1)$ with $\|K_1\|_1<\frac{\varepsilon}{4}$ such that $(A_1+iB_1)+K_1$ is irreducible in $\mathcal{B}(\mathcal{H}_1)$.
	Moreover, we can require that $\alpha\notin\sigma(A_1+\operatorname{Re}K_1)$.

	{\bf Step 3.}
	If $\mathcal{H}_2=\{0\}$, then let $K_2=0$.
	Otherwise, by \Cref{conj II1-factor} and \Cref{prop type-II}, there is a trace-class operator $K_2$ in $\mathcal{B}(\mathcal{H}_2)$ with $\|K_2\|_1<\frac{\varepsilon}{4}$ such that $(A_2+iB_2)+K_2$ is a direct sum of at most countably many irreducible operators, and $\alpha\notin\sigma(A_2+\operatorname{Re}K_2)$.

	{\bf Step 4.}
	With \textbf{Step 2} and \textbf{Step 3}, we obtain that the operator
	\begin{equation*}
  		\begin{pmatrix}
    		(A_1+iB_1)+K_1 & 0 \\
    		0 & (A_2+iB_2)+K_2
  		\end{pmatrix}
  	\begin{array}{l}
    	\mathcal{H}_1 \\
    	\mathcal{H}_2 \\
  	\end{array}
	\end{equation*}
	is a direct sum $(\bigoplus_{j\in J_1}X_j)\oplus(\bigoplus_{j\in J_2}Y_j)$ of at most countably many irreducible operators, where each $X_j$ acts on a finite-dimensional Hilbert space and each $Y_j$ acts on an infinite-dimensional Hilbert space.
	By \Cref{lem direct-sum-PI-VNA}, $(\bigoplus_{j\in J_2}Y_j)\oplus(A_\infty+iB_\infty)$ generates a properly infinite von Neumann algebra or vanishes.

	Note that $\alpha E\oplus\operatorname{Re}(\bigoplus_{j\in J_1}X_j)$ is diagonal and $E\ne\{0\}$. If the direct summand $(\bigoplus_{j\in J_2}Y_j)\oplus(A_\infty+iB_\infty)$ vanishes, then the proof is finished by \Cref{lem diagonalizable}. Otherwise,
	by applying \Cref{lemma pre-main}, there is a trace-class operator $K_3$ with $\|K_3\|_1 < \frac{\varepsilon}{4}$ such that $T+K$ is irreducible on $\mathcal{H}$, where $K=K_0+K_1+K_2+K_3$ and  $\|K\|_1<\varepsilon$.
	This completes the proof.
\end{proof}

\section*{Declarations}
\noindent{\bf Conflict of interest}
On behalf of all authors, the corresponding author states that there is no conflict of interest.

\section*{Acknowledgments}
The authors gratefully acknowledge the anonymous referees for their insightful comments and constructive suggestions, as well as research support from Dalian University of Technology during the preparation of this manuscript.

\end{document}